\documentclass[11pt,reqno]{amsart}
\usepackage{amsmath}
\newcommand{\p}[2]{\gamma^{#2}_{#1}}
\newcommand{\C}[1]{\mathcal{#1}}
\newcommand{\m}[2]{#1\stackrel{\bullet}{~}#2}

\newtheorem{theorem}{Theorem}[section]
\newtheorem{lemma}[theorem]{Lemma}
\newtheorem{proposition}[theorem]{Proposition}
\newtheorem{corollary}[theorem]{Corollary}

\theoremstyle{definition}

\newtheorem{definition}[theorem]{Definition}

\newtheorem{remark}[theorem]{Remark}
\newtheorem{notation}[theorem]{Notation}
\newtheorem{example}[theorem]{Example}

\begin{document}

\markboth{ }{ }

\title{A Categorical Approach to Groupoid Frobenius Algebras}

\author{David N. Pham}

\address{Department of Mathematics\\
Marymount Manhattan College\\
221 E 71st Street, NY, NY 10021
\\
email: dpham90@gmail.com, dpham@mmm.edu}


\maketitle

\begin{abstract}
In this paper, we show that $\C{G}$-Frobenius algebras (for $\C{G}$ a finite groupoid) correspond to a particular class of Frobenius objects in the representation category of $D(k[\C{G}])$, where $D(k[\C{G}])$ is the Drinfeld double of the quantum groupoid $k[\C{G}]$ \cite{NTV}.   
\end{abstract}

\section{Introduction}	
Groupoid Frobenius algebras were introduced recently in \cite{P} as a groupoid version of (non-projective) $G$-Frobenius algebras ($G$-FAs) for $G$ a finite group \cite{Tu} \cite{K2}.  As shown\footnote{For an alternate approach to $G$-FAs, see \cite{K2}.} in \cite{Tu}, $G$-FAs are the algebraic structures which classify certain homotopy quantum field theories (HQFTs).  Roughly speaking, a $(d+1)$-dimensional HQFT is a topological quantum field theory \cite{At} for $d$-dimensional manifolds and $(d+1)$-dimensional cobordisms endowed with homotopy classes of maps into a given space $X$.  In the case when $X$ is an Eilenberg-MacLane space of type $K(G,1)$, one finds that the associated $(1+1)$-dimensional HQFTs are classified by $G$-FAs \cite{Tu}.

The author's original motivation for generalizing $G$-FAs to $\C{G}$-FAs for $\C{G}$ a finite groupoid was the appearance of certain ``atypical" $G$-FAs in \cite{KP}, which were constructed within the framework of stringy orbifold theory (cf \cite{JKK1} \cite{JKK2} \cite{FG}).  To get a basic idea of this construction, let $M$ be a compact, almost complex manifold with an action by a finite group $G$ which preserves the almost complex structure.  Let $I(M)$ denote the \textit{inertia manifold} of $M$, that is,
\begin{equation}
I(M):=\bigsqcup_{g\in G} M^g,
\end{equation}
where $M^g:=\{m\in M~|~g\cdot m=m\}$.  Let 
\begin{equation}
\mathcal{H}(M,G):=\bigoplus_{g\in G} H^{ev} (M^g),
\end{equation}
where $H^{ev} (M^g)$ denotes the even part of the ordinary cohomology of $M^g$ with rational coefficients.  Then $\mathcal{H}(M,G)$ can be endowed with a $G$-graded product, a $G$-action, and a $G$-invariant bilinear form which turns $\mathcal{H}(M,G)$ into a $G$-FA; $\mathcal{H}(M,G)$ together with the aforementioned $G$-graded product is called the \textit{stringy cohomology ring} of the $G$-manifold $M$.  To see how $\C{G}$-FAs arise from all this, we look to the inertial manfiold $I(M)$ which has a natural $G$-action given by
\begin{equation}
h\cdot (g,m):=(hgh^{-1},hm).
\end{equation}  
If one takes the stringy cohomology of $I(M)$ with its natural $G$-action, one obtains a $G$-FA with some additional structure; this additional structure is precisely that of a groupoid Frobenius algebra.  More specifically, $\mathcal{H}(I(M),G)$ turns out to be a $\wedge\overline{G}$-FA, where $\wedge\overline{G}$ is the \textit{loop groupoid} of the one object groupoid associated with $G$.

The existence of these atypical $G$-FAs motivated the view that $G$-FAs are actually a special case of some larger algebraic structure.  Ultimately, it was the transition from group to groupoid that resulted in a framework that was capable of accommodating these atypical $G$-FAs.   As it turned out, these early motivating examples were just the tip of the iceberg.   It was shown in \cite{P} that by working within the $\C{G}$-FA framework, one could construct a tower of increasingly complex $G$-FAs, where each $G$-FA in the tower is derived from some groupoid Frobenius algebra.   In addition to this, $\C{G}$-FAs could also be used to gain new insight on the problem of twisting ordinary $G$-FAs. 

It was shown in \cite{K1} that every $G$-FA has a twist by any element of $Z^2(G,k^\times)$.  Since these twists apply to all $G$-FAs, one can regard them as  ``universal" $G$-FA twists.  In an analogous manner, $\C{G}$-FAs have their own universal twists where the twisting is now by the elements of $Z^2(\C{G},k^\times)$ \cite{P}.  When one combines this point with the aforementioned tower of ``$\C{G}$-FA induced" $G$-FAs, one  obtains a significant generalization of the $G$-FA twisting result from \cite{K1}.  Specifically, for every $n\ge 2$, one can always find a class of $G$-FAs with twists by any element in $Z^n(G,k^\times)$ \cite{P}.  

While \cite{P} illustrates the utility of $\C{G}$-FAs in addressing and solving these problems, little was done in \cite{P} to motivate the choice of axioms for a $\C{G}$-FA.   The only motivation for the axioms came in the form of a short remark\footnote{More specifically, Remark 3.1 of \cite{P}.} which asserted that $\C{G}$-FAs might actually correspond to certain kinds of Frobenius objects in $\mbox{Rep}(D(k[\C{G}]))$ (the representation category of $D(k[\C{G}])$), where $D(k[\C{G}])$ is the Drinfeld double of the quantum groupoid (weak Hopf algebra) $k[\C{G}]$ \cite{NTV}.  This assertion is motivated by a recent categorical result for $G$-FAs \cite{KP} which showed that $G$-FAs correspond to certain kinds of Frobenius objects in $\mbox{Rep}(D(k[G]))$, where $D(k[G])$ is the original Drinfeld double of the Hopf algebra $k[G]$ \cite{Drin}.  Consequently, if the assertion proves true,  the $\C{G}$-FA axioms of \cite{P} would essentially be a consequence of generalizing $D(k[G])$ to $D(k[\C{G}])$.  In other words, from this categorical vantage point, the notion of a $\C{G}$-FA is a natural generalization of a $G$-FA for the case when $G$ is replaced by $\C{G}$.   With the current paper, we show that the assertion of \cite{P} is indeed true.    

The rest of the paper is organized as follows.  In section 2, we give a brief review of quantum groupoids \cite{BS} \cite{BS1} \cite{N} and their representation category \cite{NTV}.  In section 3, we  prove the assertion raised in \cite{P}.  We conclude the paper in section 4 with some open questions.
\section{Preliminaries}
Throughout this paper, we use the following notation.
\begin{itemize}
\item[] $k$ is a field of characteristic $0$.
\vspace*{0.05in}
\item[] $\C{G}=(\C{G}_0,\C{G}_1)$ denotes a finite groupoid whose set of objects is $\C{G}_0$ and whose set of morphisms is $\C{G}_1$. 
\vspace*{0.05in}
\item[ ] The source and target maps from $\C{G}_1$ to $\C{G}_0$ are denoted as $s$ and $t$ respectively.
\vspace*{0.05in}
\item[ ] For $\mathrm{x}\in \C{G}_0$, $e_\mathrm{x}$ denotes the identity morphism associated to $\mathrm{x}$.
\vspace*{0.05in} 
\item[ ]  For $\mathrm{x}\in \C{G}_0$, $\Gamma^\mathrm{x}$ is the group consisting of all $g\in \C{G}_1$ with $s(g)=t(g)=\mathrm{x}$.
\end{itemize}
\subsection{Quantum Groupoids}
Quantum groupoids or weak Hopf algebras \cite{BS} \cite{BS1} \cite{N} generalize the notion of ordinary Hopf algebras by weakening the axioms concerning the coproduct and counit in the following way:
\begin{itemize}
\item[1.] the coproduct is not necessarily unit-preserving;
\item[2.] the counit is not necessarily multiplicative.
\end{itemize}
Formally, a quantum groupoid is defined as follows:
\begin{definition}
\label{QuantumGroupoidDef}
A quantum groupoid over a field $k$ is a tuple $(H,\cdot,1,\Delta,\varepsilon,S)$ where  
\begin{itemize}
\item[(i)] $H$ is a finite dimensional unital associative algebra over $k$ with product $\cdot$ 
 and unit $1$.
\item[(ii)] $H$ is a finite dimensional counital coassociative algebra over $k$ with coproduct $\Delta: H\rightarrow H\otimes_k H$ and counit $\varepsilon: H\rightarrow k$.
\item[(iii)] The algebra and coalgebra structure of $H$ satisfy the following compatibility conditions.
\begin{itemize}
\item[(a)] Multiplicativity of the coproduct:  for all $x,y\in H$, 
\begin{align}
\nonumber
\Delta(x\cdot y)=\Delta(x)\cdot \Delta(y)
\end{align}
\item[(b)] Weak multiplicativity of the counit: for all $x,y,z\in H$,
\begin{align}
\nonumber
\varepsilon(x\cdot y\cdot z)&=\varepsilon(x\cdot y_{(1)})\varepsilon(y_{(2)}\cdot z)\\
\nonumber
\varepsilon(x\cdot y\cdot z)&=\varepsilon(x\cdot y_{(2)})\varepsilon(y_{(1)}\cdot z)
\end{align}
\item[(c)] Weak comultiplicativity of the unit:
\begin{align}
\nonumber
(\Delta\otimes id_H)\circ \Delta(1)&=(\Delta(1)\otimes 1)\cdot (1\otimes \Delta(1))\\
\nonumber
(\Delta\otimes id_H)\circ \Delta(1)&=(1\otimes \Delta(1))\cdot (\Delta(1)\otimes 1)
\end{align}
\end{itemize}
\item[(iv)] $S:H\rightarrow H$ is a $k$-linear map called the antipode which satisifes the following for all $x\in H$:
\begin{itemize}
\item[(a)] $x_{(1)}\cdot S(x_{(2)})=\varepsilon(1_{(1)}\cdot x)1_{(2)}$
\item[(b)] $S(x_{(1)})\cdot x_{(2)}=1_{(1)}\varepsilon(x\cdot 1_{(2)})$
\item[(c)] $S(x_{(1)})\cdot x_{(2)}\cdot S(x_{(3)})=S(x)$
\end{itemize}
\end{itemize}
\end{definition}
\begin{remark}
In Definition \ref{QuantumGroupoidDef}, Sweedler notation was applied so that $\Delta(a)$ is written as $\Delta(a)=a_{(1)}\otimes a_{(2)}$.
\end{remark}
\begin{remark}
\label{QuantumGrpdRemark1}
Its a straightforward exercise to show the following:
\begin{itemize}
\item[1.] Every Hopf algebra is a quantum groupoid.
\item[2.] For a quantum groupoid $H$, the following statements are equivalent:
\begin{itemize}
\item[(i)] $H$ is a Hopf algebra
\item[(ii)] $\Delta(1)=1\otimes 1$
\item[(iii)] $\varepsilon(x\cdot y)=\varepsilon(x)\varepsilon(y)$ for all $x,y\in H$
\end{itemize}
\end{itemize}
\end{remark}
\begin{example}
Any finite groupoid $\C{G}$ defines a quantum groupoid
\begin{equation}
\nonumber
(k[\C{G}],\cdot,1,\Delta,\varepsilon,S)
\end{equation}
where 
\begin{itemize}
\item[1.] $k[\C{G}]:=\bigoplus_{g\in \C{G}_1} kg$ as a vector space over $k$
\item[2.] $\cdot$ is the multiplication on $k[\C{G}]$ induced by the composition of morphisms, that is, for $g,h\in \C{G}_1$, $g\cdot h=gh$ if $s(g)=t(h)$ and $g\cdot h=0$ if $s(g)\neq t(h)$
\item[3.] $1:=\sum_{\mathrm{x}\in \C{G}_0} e_\mathrm{x}$
\item[4.] $\Delta: k[\C{G}]\rightarrow k[\C{G}]\otimes_k k[\C{G}]$ is the $k$-linear map induced by $g\mapsto g\otimes g$ for all $g\in \C{G}_1$
\item[5.] $\varepsilon: k[\C{G}]\rightarrow k$ is the $k$-linear map induced by $g\mapsto 1_k$ for all $g\in \C{G}_1$ where $1_k$ is the unit element of $k$
\item[6.] $S: k[\C{G}]\rightarrow k[\C{G}]$ is the $k$-linear map induced by $g\mapsto g^{-1}$ for all $g\in \C{G}_1$.
\end{itemize}
\end{example}
\noindent We conclude this section by recalling a few things about \textit{quasitriangular} quantum groupoids \cite{NTV}; we begin with its definition.
\begin{definition}
\label{QuasitriangularQuantumDef}
A quasitriangular quantum groupoid is a tuple $(H,\cdot,1,\Delta,\varepsilon,S,R)$ where 
\begin{itemize}
\item[(i)] $(H,\cdot,1,\Delta,\varepsilon,S)$ is a quantum groupoid, and
\item[(ii)] $R\in \Delta^{op}(1)(H\otimes_k H)\Delta(1)$ satisfies the following conditions for all $h\in H$:
\begin{align}
\Delta^{op}(h)R&=R\Delta(h)\\
(id_H\otimes \Delta)(R)&=R_{13}\cdot R_{12}\\
(\Delta\otimes id_H)(R)&=R_{13}\cdot R_{23}
\end{align}
where $\Delta^{op}$ is the opposite coproduct, $R_{12}=R\otimes 1$, $R_{23}=1\otimes R$, and $R_{13}=R^{(1)}\otimes 1\otimes R^{(2)}$.  In addition, there exists $\overline{R}\in \Delta(1)(H\otimes_k H)\Delta^{op}(1)$ such that 
\begin{align}
R\cdot \overline{R}&=\Delta^{op}(1)\\
\overline{R}\cdot R&=\Delta(1).
\end{align}
\end{itemize}
\end{definition}
\begin{remark}
In Definition \ref{QuasitriangularQuantumDef}, the R-matrix $R$ was written as
\begin{equation}
\nonumber
R=R^{(1)}\otimes R^{(2)}
\end{equation}
to simplify notation.
\end{remark}
A Drinfeld double construction was introduced in \cite{NTV} for generating quasitriangular quantum groupoids from existing quantum groupoids.  When this construction is applied to the quantum groupoid $k[\mathcal{G}]$, the result is the quasitriangular quantum groupoid $D(k[\mathcal{G}])$ which is defined as follows:
\begin{itemize}
\item[1.] As a vector space over $k$, $D(k[\mathcal{G}])$ has basis 
\begin{equation}
\{\p{g}{x}~|~g,x\in \mathcal{G}_1,~s(g)=t(g)=t(x)\}.
\end{equation}
\item[2.] For $\p{g}{x},~\p{h}{y}\in D(k[\mathcal{G}])$, the multiplication law is given by
\begin{equation}
\label{DkG_product}
\p{g}{x}\cdot \p{h}{y}:=\delta_{x^{-1}gx, h}~\p{g}{xy}.
\end{equation} 
(Note that when $x^{-1}gx=h$,  $xy$ is defined since $s(x)=s(h)=t(y)$).
\item[3.] The unit of $D(k[\mathcal{G}])$ is 
\begin{equation}
1=\sum_{\mathrm{x}\in \C{G}_0} 1^\mathrm{x}
\end{equation}
where
\begin{equation}
\label{1x}
1^\mathrm{x}:=\sum_{g\in \Gamma^\mathrm{x}}\p{g}{e_\mathrm{x}}.
\end{equation}
\item[4.] The coproduct of $D(k[\mathcal{G}])$ is defined as
\begin{equation}
\triangle_D(\p{g}{x}):= \sum_{\{g_1,g_2\in \Gamma^{t(x)}~|~g_1g_2=g\}} \p{g_1}{x}\otimes \p{g_2}{x}
\end{equation}
\item[5.] The counit of $D(k[\mathcal{G}])$ is defined as
\begin{equation}
\varepsilon_D(\p{g}{x})=\delta_{g,xx^{-1}}.
\end{equation}
\item[6.] The antipode of $D(k[\mathcal{G}])$ is defined as
\begin{equation}
S(\p{g}{x})=\p{x^{-1}g^{-1}x}{x^{-1}}
\end{equation}
\item[7.] The R-matrix is 
\begin{equation}
\label{RmatrixDkG}
R:=\sum_{\mathrm{x}\in \mathcal{G}_0} R^\mathrm{x}
\end{equation}
where 
\begin{equation}
\label{Rx}
R^\mathrm{x}:=\sum_{g,h\in \Gamma^{\mathrm{x}}} \p{g}{e_\mathrm{x}}\otimes \p{h}{g}.
\end{equation}
\end{itemize}
\begin{remark}
\label{coproduct_unit}
Note that unless $\C{G}$ has a single object, $\Delta_D$ does not preserve the unit since
\begin{align}
\label{coproduct_unit1}
\Delta_D(1^\mathrm{x})&=\sum_{g\in \Gamma^\mathrm{x}}\Delta_D(\p{g}{e_\mathrm{x}})=\sum_{g\in \Gamma^{\mathrm{x}}}~\sum_{\{g_1,g_2\in \Gamma^\mathrm{x}~|~g_1g_2=g\}}\p{g_1}{e_{\mathrm{x}}}\otimes\p{g_2}{e_\mathrm{x}}=1^\mathrm{x}\otimes1^\mathrm{x}
\end{align}
and 
\begin{align}
\label{coproduct_unit2}
\Delta_D(1)=\sum_{\mathrm{x}\in \C{G}_0} \Delta_D(1^\mathrm{x})= \sum_{\mathrm{x}\in \C{G}_0}1^\mathrm{x}\otimes 1^\mathrm{x}\neq \sum_{\mathrm{x},\mathrm{y}\in \C{G}_0}1^\mathrm{x}\otimes 1^\mathrm{y}=1\otimes 1.
\end{align}
So by Remark \ref{QuantumGrpdRemark1}, $D(k[\C{G}])$ is a Hopf algebra only when $\C{G}$ is a one-object groupoid (i.e., a group).  In the case when $\C{G}$ is the one-object groupoid whose set of morphisms is the group $G$, $D(k[\C{G}])$ is exactly $D(k[G])$, the ordinary Drinfeld double of the Hopf algebra $k[G]$.
\end{remark}

\subsection{Quantum Groupoids $\&$ Category Theory}
It was shown in \cite{NV} that for a quantum groupoid $H$, $\mbox{Rep}(H)$\footnote{$\mbox{Rep}(H)$ is the category whose objects are finite dimensional left $H$-modules and whose morphisms are $H$-linear maps.} is a monoidal category.   To define the monoidal product, let $(\rho_1,A_1)$ and $(\rho_2,A_2)$ be objects of $\mbox{Rep}(H)$.   Then
\begin{equation}
 (\rho_1,A_1)\otimes (\rho_2,A_2):=(\rho_{12},A_1\widehat{\otimes}A_2)
 \end{equation}
where the $H$-action $\rho_{12}$ is induced by the coproduct $\Delta$ of $H$ via
\begin{equation}
\label{action1}
\rho_{12}(h):=[\rho_1\otimes \rho_2]\circ \Delta(h),~h\in H
\end{equation}
and 
\begin{align}
\nonumber
A_1\widehat{\otimes}A_2&:=\{a\in A_1\otimes_k A_2~|~\rho_{12}(1)a=a\}\\
\nonumber
&=\rho_{12}(1)(A_1\otimes_k A_2)\\
\label{tensor_product1}
&=[\rho_1(1_{(1)})\otimes \rho_2(1_{(2)})]\left(A_1\otimes_k A_2\right)
\end{align}
where the second equality follows from the fact that $\Delta(1)\cdot \Delta(1)=\Delta(1)$.  The monoidal product of morphisms is simply the restriction of the usual tensor product  of linear maps.

For the unit object, let $\varepsilon_t:H\rightarrow H$ be defined by
\begin{equation}
\label{epsilon_t}
\varepsilon_t(h):=\varepsilon(1_{(1)}\cdot h)1_{(2)}
\end{equation}
where $h\in H$ and $\varepsilon$ is the counit of $H$.  Then the unit object of $\mbox{Rep}(H)$ is $I=(\sigma_t,H_t)$ where 
\begin{equation}
H_t:=\varepsilon_t(H),
\end{equation}
and for $h\in H$ and $z\in H_t$,
\begin{equation}
\sigma_t(h)z:=\varepsilon_t(h\cdot z).
\end{equation}

If $(\rho_i,A_i)$ are objects of $\mbox{Rep}(H)$ for $i=1$, $2$, and $3$, then the associator 
\begin{equation}
\Phi_{123}: (A_1\widehat{\otimes} A_2)\widehat{\otimes}A_3\stackrel{\sim}{\longrightarrow} A_1\widehat{\otimes} (A_2\widehat{\otimes}A_3)
\end{equation}
is the trivial one.

To define the left$\backslash$right unit morphisms, let $(\rho,A)$ be an object of $\mbox{Rep}(H)$.  Then the left morphism 
\begin{equation}
l_A: H_t\widehat{\otimes} A\stackrel{\sim}{\longrightarrow} A
\end{equation}
is defined by
\begin{equation}
\label{l_Aaction}
l_A\left(\sigma_t(1_{(1)})z\otimes\rho(1_{(2)})a\right):=\rho(z)a
\end{equation}
where $z\in H_t$ and $a\in A$; the right morphism
\begin{equation}
r_A: A\widehat{\otimes} H_t\stackrel{\sim}{\longrightarrow} A
\end{equation}
is defined by 
\begin{equation}
\label{r_Aaction}
r_A\left(\rho(1_{(1)})a\otimes \sigma_t(1_{(2)})z\right):=\rho(S(z))a
\end{equation}
where $z\in H_t$, $a\in A$, and $S$ is the antipode of $H$.

If $H$ is also quasitriangular with R-matrix $R$, then $\mbox{Rep}(H)$ is also braided \cite{NTV}.  For any two objects $(\rho_1,A_1)$ and $(\rho_2,A_2)$ of $\mbox{Rep}(H)$, the braiding
\begin{equation}
c_{A_1,A_2}: A_1\widehat{\otimes} A_2\rightarrow A_2\widehat{\otimes} A_1
\end{equation}
 is defined by
 \begin{equation}
 \label{BraidingDef}
 c_{A_1,A_2}(x):=\rho_2(R^{(2)})x^{(2)}\otimes \rho_1(R^{(1)})x^{(1)}
 \end{equation}
where $x=x^{(1)}\otimes x^{(2)}\in A_1\widehat{\otimes} A_2$.
\subsubsection{Frobenius Objects}
Throughout this section, $\left(\C{C},\otimes,I,\Phi,l,r,c\right)$ will denote a braided monoidal category where $\C{C}$ is a small category, $\otimes$ is the monoidal product, $I$ is the unit object, $\Phi$ is the associator, $l$ and $r$ are the left and right identity maps, and $c$ is the braiding.
\vspace*{0.1in} 
\begin{definition}
\label{AlgebraObjectDef}
An algebra object is a tuple $\left(A,m,\mu\right)$ where 
\begin{itemize}
\item[] $A$ is an object of $\C{C}$, 
\item[] $m: A\otimes A\rightarrow A$ is a morphism of $\C{C}$ called the product, and
\item[] $\mu: I\rightarrow A$ is a morphism of $\C{C}$ called the unit
\end{itemize}
which satisfy the following two conditions:
\begin{itemize}
\item[1.] $m\circ (id_A\otimes m)\circ \Phi_{A,A,A}=m\circ (m\otimes id_A)$  (associativity)
\item[2.] $m\circ(\mu\otimes id_A)=l_A$, $m\circ(id_A\otimes \mu)=r_A$ (unit property)
\end{itemize}
$\left(A,m,\mu\right)$ is said to be commutative if $m\circ c_{A,A}=m$.
\end{definition}
\begin{definition}
\label{CoalgebraObjectDef}
A coalgebra object is a tuple $\left(A,\Delta,\varepsilon\right)$ where 
\begin{itemize}
\item[] $A$ is an object of $\C{C}$, 
\item[] $\Delta: A\rightarrow A\otimes A$ is a morphism of $\C{C}$ called the coproduct, and
\item[] $\varepsilon: A\rightarrow I$ is a morphism of $\C{C}$ called the counit
\end{itemize}
which satisfy the following two conditions:
\begin{itemize}
\item[1.] $(id_A\otimes \Delta)\circ\Delta=\Phi_{A,A,A}\circ(\Delta\otimes id_A)\circ \Delta$ (coassociativity)
\item[2.] $l_A\circ(\varepsilon\otimes id_A)\circ\Delta=id_A=r_A\circ(id_A\otimes \varepsilon)\circ \Delta$ (counit property)
\end{itemize}
$\left(A,\Delta,\varepsilon\right)$ is said to be co-commutative if $c_{A,A}\circ\Delta=\Delta$.
\end{definition}
\begin{definition}
\label{FrobObjDef}
A Frobenius object is a tuple $\left(A,m,\Delta,\mu,\varepsilon\right)$ where $\left(A,m,\mu\right)$ is a commutative algebra object and $\left(A,\Delta,\varepsilon\right)$ is a co-commutative coalgebra object which satisfies
\begin{align}
\label{FrobeniusAxiom1}
\Delta\circ m&=(m\otimes id_A)\circ \Phi_{A,A,A}^{-1}\circ(id_A\otimes \Delta)\\
\label{FrobeniusAxiom2}
\Delta\circ m&=(id_A\otimes m)\circ \Phi_{A,A,A}\circ(\Delta\otimes id_A).
\end{align} 
\end{definition}
\begin{remark}
Throughout this paper, we will disregard $\Phi$ from our expressions since $\mbox{Rep}(D(k[\C{G}]))$ (our category of interest) has a trivial associator.
\end{remark}
\section{$\C{G}$-FAs \& Frobenius Objects in $\mbox{Rep}(D(k[\C{G}]))$}
We begin this section by recalling the axiomatic definition of a $\C{G}$-FA \cite{P}:
\begin{definition}
\label{DefGrpdFA}
A $\mathcal{G}$-Frobenius Algebra ($\mathcal{G}$-FA) is given by the following data
\begin{equation}
\nonumber
<\mathcal{G},(A,\m{}{},\textbf{1}_A),\eta,\varphi>
\end{equation}
where
\begin{itemize}
\item[(a)] $\mathcal{G}=(\mathcal{G}_0,\mathcal{G}_1)$ is a finite groupoid.
\item[(b)] $(A,\m{}{},\textbf{1}_A)$ is a finite dimensional associative algebra over $k$ with product $\m{}{}$ and unit $\textbf{1}_A$ which splits as a direct sum of algebras which are indexed by the objects of $\mathcal{G}$:
\begin{equation}
A=\bigoplus_{\mathrm{x}\in \mathcal{G}_0}A^{\mathrm{x}}.
\end{equation}
\item[(c)] $\eta: A\times A\rightarrow k$ is a bilinear form.
\item[(d)] $\varphi$ is a $\mathcal{G}$-action which acts on $A$ by algebra homomorphisms, that is, if (1) $x\in \mathcal{G}_1$ with $s(x)=\mathrm{x}$ and $t(x)=\mathrm{y}$, then $\varphi(x): A^\mathrm{x}\rightarrow A^\mathrm{y}$ is an algebra isomorphism, (2) if $g,h\in \mathcal{G}_1$ and $s(h)=t(g)$, then $\varphi(h)\circ \varphi(g)=\varphi(hg)$, and (3) $\varphi(e_\mathrm{x})=id_{A^\mathrm{x}}$
\end{itemize}
\textit{which satisfies the following for all $\mathrm{x},~\mathrm{y}\in \mathcal{G}_0$}: 
\begin{itemize}
\item[(i)] $A^\mathrm{x}=\bigoplus_{g\in \Gamma^\mathrm{x}}A^\mathrm{x}_g$ is a $\Gamma^\mathrm{x}$-graded algebra.  
\item[(ii)] if $a^{\mathrm{x}}\in A^\mathrm{x}$ and $b^{\mathrm{y}}\in A^\mathrm{y}$, then $\m{a^{\mathrm{x}}}{b^{\mathrm{y}}}=\delta_{\mathrm{x},\mathrm{y}}~\m{a^\mathrm{x}}{b^\mathrm{y}}\in A^\mathrm{x}$.   
\item[(iii)] $\eta(\m{a}{b},c)=\eta(a,\m{b}{c})$ for all $a,b,c\in A$.
\item[(iv)] $\eta(\varphi(h)a^\mathrm{x},\varphi(h)b^\mathrm{x})=\eta(a^\mathrm{x},b^\mathrm{x})$ for all $a^\mathrm{x},b^\mathrm{x}\in A^\mathrm{x}$ and $h\in \mathcal{G}_1$ satisfying $s(h)=\mathrm{x}$.
\item[(v)] $\varphi(x)A^\mathrm{x}_g\subset A^{\mathrm{y}}_{xgx^{-1}}$ for a morphism $x\in \C{G}_1$ satisfying $s(x)=\mathrm{x}$ and $t(x)=\mathrm{y}$.
\item[(vi)] $\eta|_{A^\mathrm{x}_g\times A^\mathrm{x}_h}$ is nondegenerate for $gh=e_\mathrm{x}$ and zero otherwise.
\item[(vii)] $\m{a_g^\mathrm{x}}{a_h^\mathrm{x}}=\m{\left(\varphi(g)a^\mathrm{x}_h\right)}{a^\mathrm{x}_g}\in A^\mathrm{x}_{gh}$ for $a_g^\mathrm{x}\in A^\mathrm{x}_g$ and $a^\mathrm{x}_h\in A^\mathrm{x}_h$. 
\item[(viii)] $\varphi(g)|_{A^\mathrm{x}_g}=id_{A^\mathrm{x}_g}$
\item[(ix)] if $g,h\in \Gamma^\mathrm{x}$, $c\in A^\mathrm{x}_{ghg^{-1}h^{-1}}$, and $l_c:A\rightarrow A$ is the linear map induced by left multiplication by $c$, then
\begin{equation}
\mbox{Tr}\left(l_c\circ \varphi(h)|_{A^\mathrm{x}_g}: A^\mathrm{x}_g\rightarrow A^\mathrm{x}_g\right)= \mbox{Tr}\left(\varphi(g^{-1})\circ l_c|_{A^\mathrm{x}_h}: A^\mathrm{x}_h\rightarrow A^\mathrm{x}_h\right)
\end{equation}
where Tr denotes the trace.
\end{itemize}
\end{definition}
\begin{remark}
The definition of groupoid Frobenius algebras given in \cite{P} was stated in terms of group Frobenius algebras.  In an effort to make Definition \ref{DefGrpdFA} self contained, we have reworded the original definition of \cite{P} to avoid any reference to group Frobenius algebras. 
\end{remark}
\begin{remark}
In the special case when $\C{G}$ is the one-object groupoid whose set of morphisms is the group $G$, Definition \ref{DefGrpdFA} reduces to the definition of a $G$-Frobenius algebra.
\end{remark}
\noindent We now state the main result of this paper:
\begin{theorem}
\label{MainTheorem}
Every $\C{G}$-FA is derived from a Frobenius object $\left((\rho,A),m,\Delta,\mu,\varepsilon\right)$ in $\mbox{Rep}\left(D(k[\C{G}])\right)$ which satisfies 
\begin{itemize}
\item[(1)] $\sum_{\mathrm{x}\in \C{G}_0}\sum_{g\in \Gamma^\mathrm{x}} \rho(\p{g}{g})=id_A$\\
\item[(2)] $\mbox{Tr}( l_c\circ\rho(\p{hgh^{-1}}{h}))=\mbox{Tr}(\rho(\p{h}{g^{-1}})\circ l_c\circ \rho(\p{h}{e_\mathrm{x}}))$ $\forall$ $\mathrm{x}\in \C{G}_0$, $g,h\in \Gamma^\mathrm{x}$, and $c\in \rho(\p{ghg^{-1}h^{-1}}{e_\mathrm{x}})A$ where Tr denotes the trace and $l_c$ is the $k$-linear map defined by $l_c(v)=m(c\otimes v)$ for $v\in \rho(1^\mathrm{x})A$.
\end{itemize}
In addition, every Frobenius object in $\mbox{Rep}\left(D(k[\C{G}])\right)$ which satisfies conditions (1) and (2) induces a $\C{G}$-FA.
\end{theorem} 

\noindent We now dedicate the remainder of the paper to the proof of Theorem \ref{MainTheorem}. 

\subsection{$\C{G}$-FAs via Frobenius objects}
\noindent In this section, we show that every Frobenius object in $\mbox{Rep}(D(k[\C{G}]))$ which satisfies conditions (1) and (2) of Theorem \ref{MainTheorem} corresponds to a particular $\C{G}$-FA.  
\vspace*{0.05in}\\
\noindent We begin by showing that every left $D(k[\C{G}])$-module has a canonical direct sum decomposition and left $\C{G}$-action which resembles that of a $\C{G}$-FA.
\begin{proposition}
\label{DirectSum1}
Let $(\rho,A)$ be a left $D(k[\C{G}])$-module.  Then
\begin{itemize}
\item[(i)] $A$ has a direct sum decomposition $A=\bigoplus_{\mathrm{x}\in \C{G}_0}A^\mathrm{x}$ which is indexed by the objects of $\C{G}$ where $A^\mathrm{x}:=\rho(1^\mathrm{x})A$.  In particular, $\rho(1^\mathrm{y})a^\mathrm{x}=\delta_{\mathrm{x},\mathrm{y}}a^\mathrm{x}$ for $a^\mathrm{x}\in A^\mathrm{x}$.
\item[(ii)] For each $\mathrm{x}\in \C{G}_0$, $A^\mathrm{x}$ has a direct sum decompostion $A^\mathrm{x}=\bigoplus_{g\in \Gamma^\mathrm{x}}A^\mathrm{x}_g$ where $A^\mathrm{x}_g:=\rho(\p{g}{e_\mathrm{x}})A^\mathrm{x}=\rho(\p{g}{e_\mathrm{x}})A$.  In particular, $\rho(\p{h}{e_\mathrm{y}})a^\mathrm{x}_g=\delta_{g,h}a^\mathrm{x}_g$ for $a^\mathrm{x}_g\in A^\mathrm{x}_g$, $\mathrm{y}\in \C{G}_0$, and $h\in \Gamma^\mathrm{y}$.
\item[(iii)] For all $\mathrm{x},\mathrm{y}\in\C{G}_0$, $g\in \Gamma^\mathrm{x}$, $h\in \Gamma^\mathrm{y}$, and $x\in \C{G}_1$ with $t(x)=\mathrm{x}$, 
\begin{itemize}  
\item[(a)]  $\rho(\p{g}{x})a^\mathrm{y}_h=0$ for all $a^\mathrm{y}_h\in A^\mathrm{y}_h$ with $h\neq x^{-1}gx$, and
\item[(b)] $\rho(\p{g}{x})$ is a vector space isomorphism from $A^{s(x)}_{x^{-1}gx}$ to $A^{t(x)}_g$.
\end{itemize}
\end{itemize}
\end{proposition}
\begin{proof}
Since 
\begin{equation}
\label{rhopgx0}
\p{g}{e_\mathrm{x}}\cdot \p{h}{e_\mathrm{y}}=\p{h}{e_\mathrm{y}}\cdot \p{g}{e_\mathrm{x}}=\delta_{g,h}~\p{g}{e_\mathrm{x}}
\end{equation}
for $g\in \Gamma^\mathrm{x}$ and $h\in \Gamma^\mathrm{y}$, it follows from (\ref{1x}) that  
\begin{equation}
1^\mathrm{x}\cdot 1^\mathrm{y}=\sum_{g\in \Gamma^\mathrm{x}}\sum_{h\in \Gamma^\mathrm{y}} \delta_{g,h}\p{g}{e_\mathrm{x}}=\delta_{\mathrm{x},\mathrm{y}}\sum_{g\in \Gamma^\mathrm{x}}\p{g}{e_\mathrm{x}}=\delta_{\mathrm{x},\mathrm{y}}~1^\mathrm{x}.
\end{equation}
Hence,
\begin{equation}
\label{rho1x}
\rho(1^\mathrm{x})\circ\rho(1^\mathrm{y})=\delta_{\mathrm{x},\mathrm{y}}~\rho(1^\mathrm{x}).
\end{equation}
It follows from (\ref{rho1x}) and the definition of $A^\mathrm{x}$ that 
\begin{equation}
\label{projection1}
\rho(1^\mathrm{y})a^\mathrm{x}=\delta_{\mathrm{x},\mathrm{y}}~ a^\mathrm{x}
\end{equation}
 for $a^\mathrm{x}\in A^\mathrm{x}$. 
 
Since $\rho(1)=id_A$, we also have
\begin{equation}
\label{Adirectsum}
A=\rho(1)A=\sum_{\mathrm{x}\in \C{G}_0}\rho(1^\mathrm{x})A=\sum_{\mathrm{x}\in \C{G}_0}A^\mathrm{x}.
\end{equation}
In addition,  if $\sum_{\mathrm{x}\in \C{G}_0}a^\mathrm{x}=0$ for $a^\mathrm{x}\in A^\mathrm{x}$, it follows from (\ref{projection1}) that
\begin{equation}
\label{Adirectsum1}
a^\mathrm{y}=\rho(1^\mathrm{y})\left(\sum_{\mathrm{x}\in \C{G}_0}a^\mathrm{x}\right)=0
\end{equation}
for all $\mathrm{y}\in \C{G}_0$.  (\ref{Adirectsum}) and (\ref{Adirectsum1}) then show that $A$ is a direct sum of the $A^\mathrm{x}$'s.  This completes the proof of (i).  

For (ii), note that by (\ref{projection1})
\begin{equation}
\label{Axdirectsum}
A^\mathrm{x}=\rho(1^\mathrm{x})A^\mathrm{x}=\sum_{g\in \Gamma^\mathrm{x}}\rho(\p{g}{e_\mathrm{x}})A^\mathrm{x}=\sum_{g\in \Gamma^\mathrm{x}}A^\mathrm{x}_g
\end{equation}
and by (\ref{rhopgx0})
\begin{equation}
\label{rhopgx}
\rho(\p{h}{e_\mathrm{y}})\circ \rho(\p{g}{e_\mathrm{x}})=\delta_{g,h}~\rho(\p{g}{e_\mathrm{x}}).
\end{equation}
It follows from (\ref{rhopgx}) and the definition of $A^\mathrm{x}_g$ that
\begin{equation}
 \label{projection2}
 \rho(\p{h}{e_\mathrm{y}})a^\mathrm{x}_g=\delta_{g,h}~a^\mathrm{x}_g
 \end{equation}
 for $a^\mathrm{x}_g\in A^\mathrm{x}_g$.   Using (\ref{Axdirectsum}) and (\ref{projection2}) and applying an argument similar to the one used to prove that $A=\oplus_{\mathrm{x}\in \C{G}_0}A^\mathrm{x}$ shows that $A^\mathrm{x}$ itself is a direct sum of the $A^\mathrm{x}_g$'s.  In addition, we also have
\begin{equation}
A^\mathrm{x}_g:=\rho(\p{g}{e_\mathrm{x}})A^\mathrm{x}=\rho(\p{g}{e_\mathrm{x}})\left(\rho(1^\mathrm{x})A\right)=\rho(\p{g}{e_\mathrm{x}})A,
\end{equation}
where the second equality follows from the definition of $A^\mathrm{x}$ and the third equality follows from the fact that $\p{g}{e_\mathrm{x}}\cdot 1^\mathrm{x}=\p{g}{e_\mathrm{x}}$.  This completes the proof of (ii).
 
(iii-a) follows from (\ref{projection2}) and the fact that
\begin{align}
\nonumber
\rho(\p{g}{x})&=\rho(\p{g}{e_{t(x)}}\cdot\p{g}{x}\cdot\p{x^{-1}gx}{e_{s(x)}})\\
\label{projection3}
&=\rho(\p{g}{e_{t(x)}})\circ\rho(\p{g}{x})\circ\rho(\p{x^{-1}gx}{e_{s(x)}}).
\end{align}  
(\ref{projection3}) also shows that $\rho(\p{g}{x})A^{s(x)}_{x^{-1}gx}\subset A^{t(x)}_g$.   To see that $\rho(\p{g}{x})$ is also an isomorphism, note that  
\begin{align}
\rho(\p{g}{x})\circ \rho(\p{x^{-1}gx}{x^{-1}})|_{A^{t(x)}_g}=\rho(\p{g}{x}\cdot \p{x^{-1}gx}{x^{-1}})|_{A^{t(x)}_g}=\rho(\p{g}{e_{t(x)}})|_{A^{t(x)}_g}=id_{A^{t(x)}_g}
\end{align}
and
\begin{equation}
 \rho(\p{x^{-1}gx}{x^{-1}})\circ\rho(\p{g}{x})|_{A^{s(x)}_{x^{-1}gx}}=\rho(\p{x^{-1}gx}{x^{-1}}\cdot\p{g}{x})|_{A^{s{(x)}}_g}=\rho(\p{x^{-1}gx}{e_{s(x)}})|_{A^{s(x)}_{x^{-1}gx}}=id_{A^{s(x)}_{x^{-1}gx}}.
 \end{equation}
 This completes the proof of (iii-b).
\end{proof}
\begin{corollary}
\label{GroupoidActionCor}
Every left $D(k[\C{G}])$-module $(\rho,A)$ has a left $\C{G}$-action $\varphi$ which acts as a $k$-linear map on the direct sum decomposition $A=\oplus_{\mathrm{x}\in \C{G}_0} A^\mathrm{x}$ given by part (i) of Proposition \ref{DirectSum1}
where 
\begin{equation}
\label{Gaction}
\varphi(x):=\sum_{g\in \Gamma^{t(x)}} \rho(\p{g}{x})|_{A^{s(x)}}~\mbox{for}~x\in \C{G}_1.
\end{equation}
In addition, if $x\in \C{G}_1$ and $A^{s(x)}=\oplus_{h\in \Gamma^{s(x)}}A^{s(x)}_h$ is the direct sum decomposition given by part (ii) of Proposition \ref{DirectSum1}, then $\varphi(x)A^{s(x)}_g\subset A^{t(x)}_{xgx^{-1}}$.
\end{corollary}
\begin{proof}
If $y\in \C{G}_1$ with $t(y)=s(x)$, then $\varphi(x)\circ \varphi(y)=\varphi(xy)$ since
\begin{equation}
\sum_{g\in \Gamma^{t(x)}}\p{g}{x}~\cdot \sum_{h\in \Gamma^{t(y)}}\p{h}{y}=\sum_{g\in \Gamma^{t(xy)}}\p{g}{xy}.
\end{equation}
It follows from the definition of $\varphi$ and part (iii) of Proposition \ref{DirectSum1} that $\varphi(x)$ is a linear map from $A^{s(x)}$ to $A^{t(x)}$.  

Next we verify that $\varphi(e_\mathrm{x})=id_{A^\mathrm{x}}$.   To do this, let $a^\mathrm{x}\in A^\mathrm{x}$.  Then $a^\mathrm{x}$ can be uniquely decomposed as $a^\mathrm{x}=\sum_{g\in \Gamma^\mathrm{x}}a^\mathrm{x}_g$ for $a^\mathrm{x}_g\in A^\mathrm{x}_g$.  By part (ii) of Proposition \ref{DirectSum1}, we have
\begin{align}
\label{grpaction_unit}
\varphi(e_\mathrm{x})a^\mathrm{x}&=\sum_{g\in \Gamma^\mathrm{x}}\rho(\p{g}{e_\mathrm{x}})a^\mathrm{x}=\sum_{g\in \Gamma^\mathrm{x}}\rho(\p{g}{e_\mathrm{x}}) a^\mathrm{x}_g=\sum_{g\in \Gamma^\mathrm{x}} a^\mathrm{x}_g=a^\mathrm{x}.
\end{align}

To complete the proof that $\varphi$ is a $\C{G}$-action, we only need to show that $\varphi(x): A^{s(x)}\rightarrow A^{t(x)}$ is an isomorphism of vector spaces and this follows from the previous calculation since
\begin{align}
\varphi(x^{-1})\circ \varphi(x)=\varphi(x^{-1}x)=\varphi(e_{s(x)})=id_{A^{s(x)}}
\end{align}
and
\begin{align}
\varphi(x)\circ \varphi(x^{-1})=\varphi(xx^{-1})=\varphi(e_{t(x)})=id_{A^{t(x)}}.
\end{align}

Lastly, if $a^{s(x)}_g\in A^{s(x)}_g$, then
\begin{align}
\varphi(x)a^{s(x)}_g&=\sum_{h\in \Gamma^{t(x)}}\rho(\p{h}{x})a^{s(x)}_g=\rho(\p{xgx^{-1}}{x})a^{s(x)}_g\in A^{t(x)}_{xgx^{-1}}
\end{align}
by part (iii) of Proposition \ref{DirectSum1}.
$~$
\end{proof}
\begin{notation}
For an object $(\rho,A)$ of $\mbox{Rep}(D(k[\C{G}]))$, we will often suppress the $D(k[\C{G}])$-action $\rho$ and simply write $A$ for $(\rho,A)$.  The action of $h\in D(k[\C{G}])$ on $a\in A$ will be denoted as $h\rhd a$ when $\rho$ is omitted, that is, $h\rhd a:=\rho(h)a$.   Furthermore, for $\mathrm{x}\in \C{G}_0$ and $g\in \Gamma^\mathrm{x}$, $A^\mathrm{x}$ will denote the direct summand of $A$ given by part (i) of Proposition \ref{DirectSum1}, and $A^\mathrm{x}_g$ will denote the direct summand of $A^\mathrm{x}$ given by part (ii) of Proposition \ref{DirectSum1}.
\end{notation}
\noindent We now look at the monoidal structure of $\mbox{Rep}(D(k[\C{G}]))$, which is given by the next two lemmas.
\begin{lemma}
\label{tensor_product2a}
If $A$ and $B$ are objects of $\mbox{Rep}(D(k[\C{G}]))$, then their monoidal product $($with $D(k[\C{G}])$-action induced by the coproduct $\Delta_D$ of $D(k[\C{G}])$$)$ is 
\begin{equation}
\label{tensor_product2}
A\widehat{\otimes} B=\bigoplus_{\mathrm{x}\in \mathcal{G}_0}A^{\mathrm{x}}\otimes_k B^{\mathrm{x}}.
\end{equation}
In addition, for $\mathrm{y}\in \C{G}_0$
\begin{equation}
\label{tensor_product3}
(A\widehat{\otimes} B)^\mathrm{y}=A^\mathrm{y}\otimes_k B^\mathrm{y}.
\end{equation}
\end{lemma}
\begin{proof}
By (\ref{tensor_product1}), $A\widehat{\otimes} B:=(1_{(1)}\rhd A)\otimes_k (1_{(2)}\rhd B)$.  It follows from (\ref{coproduct_unit2}) and part (i) of Proposition \ref{DirectSum1} that 
\begin{equation}
A\widehat{\otimes} B=\sum_{\mathrm{x}\in \C{G}_0} (1^\mathrm{x}\rhd A)\otimes_k (1^\mathrm{x}\rhd B)=\sum_{\mathrm{x}\in \C{G}_0} A^\mathrm{x}\otimes_k B^\mathrm{x}.
\end{equation}
(\ref{tensor_product2}) then follows from the fact that $A=\bigoplus_{\mathrm{x}\in\C{G}_0}A^\mathrm{x}$ and $B=\bigoplus_{\mathrm{x}\in \C{G}_0}B^\mathrm{x}$.

Lastly, note that since the $D(k[\C{G}])$-action on $A\widehat{\otimes} B$ is induced by the coproduct of $D(k[\C{G}])$ and $\Delta_D(1^\mathrm{y})=1^\mathrm{y}\otimes 1^\mathrm{y}$ by (\ref{coproduct_unit1}), (\ref{tensor_product3}) follows readily from part (i) of Proposition \ref{DirectSum1} which implies that the image of $A\widehat{\otimes} B$ under $1^\mathrm{y}\otimes 1^\mathrm{y}$ is $A^\mathrm{y}\otimes_k B^\mathrm{y}$. 
\end{proof}
\begin{lemma}
\label{UnitObject}
\begin{itemize}
\item[(i)] The unit object $D(k[\C{G}])_t$ of $\mbox{Rep}(D(k[\C{G}]))$ is defined as follows:
\begin{itemize}
\item[(a)] As a vector space over $k$, $D(k[\C{G}])_t$ has basis $\{1^\mathrm{x}\}_{\mathrm{x}\in \C{G}_0}$ where $1^\mathrm{x}\in D(k[\C{G}])$ is defined by (\ref{1x}). 
\item[(b)] The left $D(k[\C{G}])$-action on $D(k[\C{G}])_t$ is given by
\begin{equation}
\p{h}{y}\rhd 1^\mathrm{x}=\delta_{s(y),\mathrm{x}}~\delta_{h,yy^{-1}}~1^{s(h)}.
\end{equation}
\end{itemize}
\item[(ii)] $D(k[\C{G}])_t^\mathrm{x}=k1^\mathrm{x}$ for $\mathrm{x}\in \C{G}_0$.  In particular,
\begin{align}
\label{IdentityObj1}
D(k[\C{G}])_t\widehat{\otimes} A &= \bigoplus_{\mathrm{x}\in\C{G}_0} k1^\mathrm{x}\otimes_k A^\mathrm{x}\\
\label{IdentityObj2}
A\widehat{\otimes} D(k[\C{G}])_t&= \bigoplus_{\mathrm{x}\in\C{G}_0}A^\mathrm{x}\otimes_k k1^\mathrm{x}
\end{align}
for an object $A$ of $\mbox{Rep}(D(k[\C{G}]))$. 
\item[(iii)] If $l_A: D(k[\C{G}])_t\widehat{\otimes} A\rightarrow A$ and $r_A: A\widehat{\otimes} D(k[\C{G}])_t\rightarrow A$ are the left$\backslash$right identity maps of $\mbox{Rep}(D(k[\C{G}]))$ for an object $A$ of $\mbox{Rep}(D(k[\C{G}]))$, then
\begin{align}
\label{l_Aaction1}
l_A(1^\mathrm{x}\otimes a^\mathrm{x})&= a^\mathrm{x}\\
\label{r_Aaction1}
r_A(a^\mathrm{x}\otimes 1^\mathrm{x})&=a^\mathrm{x}
\end{align}
for $a^\mathrm{x}\in A^\mathrm{x}$.
\end{itemize}
\end{lemma}
\begin{proof}
Let $\varepsilon_D$ and $\Delta_D$ denote the counit and coproduct of $D(k[\C{G}])$ respectively.  By definition, $D(k[\C{G}])_t$ is the image of $\varepsilon_{Dt}: D(k[\C{G}])\rightarrow D(k[\C{G}])$ where $\varepsilon_{Dt}$  is the map given by (\ref{epsilon_t}).  Since 
\begin{equation}
\Delta_D(1)=1_{(1)}\otimes 1_{(2)}=\sum_{\mathrm{x}\in \C{G}_0} 1^\mathrm{x}\otimes 1^\mathrm{x}
\end{equation}
by (\ref{coproduct_unit2}), we have
\begin{align}
\nonumber
\varepsilon_{Dt}(h)&=\varepsilon_D(1_{(1)}\cdot h)1_{(2)}\\
&=\sum_{\mathrm{x}\in \C{G}_0} \varepsilon_D(1^\mathrm{x}\cdot h)1^\mathrm{x}
\end{align}
for $h\in D(k[\C{G}])$.  Hence, the image of $\varepsilon_t$ is contained in the subspace spanned by $\{1^\mathrm{x}\}_{ \mathrm{x}\in \C{G}_0}$.  Since 
\begin{align}
\varepsilon_D(1^\mathrm{x}\cdot 1^\mathrm{z})=\delta_{\mathrm{x},\mathrm{z}}~\varepsilon_D(1^\mathrm{x})
=\delta_{\mathrm{x},\mathrm{z}}~\sum_{g\in \Gamma^\mathrm{x}}\varepsilon_D(\p{g}{e_\mathrm{x}})
=\delta_{\mathrm{x},\mathrm{z}}~\varepsilon_D(\p{e_\mathrm{x}}{e_\mathrm{x}})
=\delta_{\mathrm{x},\mathrm{z}},
\end{align}
we have $\varepsilon_{Dt}(1^\mathrm{z})=1^\mathrm{z}$.  This shows that $D(k[\C{G}])_t$ is precisely the space spanned by $\{1^\mathrm{x}\}_{\mathrm{x}\in \C{G}_0}$.  It follows from (\ref{1x}) that the latter is also linearly independent and this completes the proof of (i-a). 

For (i-b), we have
\begin{align}
\nonumber
\p{h}{y}\rhd 1^\mathrm{x}&:=\varepsilon_{Dt}(\p{h}{y}\cdot 1^\mathrm{x})\\
\nonumber
&=\delta_{s(y),\mathrm{x}}~\varepsilon_{Dt}(\p{h}{y})\\
\nonumber
&=\delta_{s(y),\mathrm{x}}~\sum_{\mathrm{z}\in \C{G}_0} \varepsilon_D(1^\mathrm{z}\cdot \p{h}{y})1^\mathrm{z}\\
\nonumber
&=\delta_{s(y),\mathrm{x}}~\varepsilon_D(\p{h}{y})1^{s(h)}\\
\nonumber
&=\delta_{s(y),\mathrm{x}}~\delta_{h,yy^{-1}}1^{s(h)}.
\end{align}

For (ii), note that if $1^\mathrm{x}\in D(k[\C{G}])$ and $1^\mathrm{y}\in D(k[\C{G}])_t$, then by (i-b), we have
\begin{equation}
\label{DkGtAction2}
1^\mathrm{x}\rhd 1^\mathrm{y}=\sum_{g\in \Gamma^\mathrm{x}} \p{g}{e_\mathrm{x}}\rhd 1^\mathrm{y}=\sum_{g\in \Gamma^\mathrm{x}}\delta_{\mathrm{x},\mathrm{y}}~\delta_{g,e_\mathrm{x}}1^{s(g)}=\delta_{\mathrm{x},\mathrm{y}}1^\mathrm{x}.
\end{equation}
Hence, it follows from this and (i-a) that 
\begin{equation}
D(k[\C{G}])_t^\mathrm{x}:=1^\mathrm{x}\rhd D(k[\C{G}])_t=k1^\mathrm{x}.
\end{equation}
(\ref{IdentityObj1}) and (\ref{IdentityObj2})  then follow from Lemma \ref{tensor_product2a}.

For (iii), we have
\begin{align}
\label{l_Aaction2}
l_A(1^\mathrm{x}\otimes a^\mathrm{x})&=\sum_{\mathrm{y}\in \C{G}_0} l_A(1^\mathrm{y}\rhd 1^\mathrm{x}\otimes 1^\mathrm{y}\rhd a^\mathrm{x})\\
\label{l_Aaction3}
&=l_A(1_{(1)}\rhd 1^\mathrm{x}\otimes 1_{(2)}\rhd a^\mathrm{x})
\end{align}
and 
\begin{align}
\label{r_Aaction2}
r_A(a^\mathrm{x}\otimes 1^\mathrm{x})&=\sum_{\mathrm{y}\in \C{G}_0} r_A(1^\mathrm{y}\rhd a^\mathrm{x}\otimes 1^\mathrm{y}\rhd 1^\mathrm{x})\\
\label{r_Aaction3}
&=r_A(1_{(1)}\rhd a^\mathrm{x}\otimes 1_{(2)}\rhd 1^\mathrm{x})
\end{align}
where (\ref{l_Aaction2}) and (\ref{r_Aaction2}) follow from (\ref{projection1}) and (\ref{DkGtAction2}), and (\ref{l_Aaction3}) and (\ref{r_Aaction3}) follow from (\ref{coproduct_unit2}).  By (\ref{l_Aaction}) and (\ref{r_Aaction}), we have
\begin{equation}
l_A(1_{(1)}\rhd 1^\mathrm{x}\otimes 1_{(2)}\rhd a^\mathrm{x})=1^\mathrm{x}\rhd a^\mathrm{x}=a^\mathrm{x}
\end{equation}
and 
\begin{equation}
r_A(1_{(1)}\rhd a^\mathrm{x}\otimes 1_{(2)}\rhd 1^\mathrm{x})=S(1^\mathrm{x})\rhd a^\mathrm{x}=1^\mathrm{x}\rhd a^\mathrm{x}=a^\mathrm{x}
\end{equation}
which proves (iii).
 \end{proof}
 \noindent We now show that a commutative algebra object in $\mbox{Rep}(D(k[\C{G}]))$ has the necessary structure to encode axioms (i), (ii), and (vii) of Definition \ref{DefGrpdFA}. 
  \begin{proposition}
 \label{ProductResult1}
Suppose $(A,m,\mu)$ is an algebra object in $\mbox{Rep}(D(k[\C{G}]))$ and
\begin{equation}
\label{product_def}
\m{a^\mathrm{x}}{b^\mathrm{y}}:=\left\{\begin{array}{cl} m(a^\mathrm{x}\otimes b^\mathrm{y})&\hspace*{0.1in}\mbox{if}\hspace*{0.1in}\mathrm{x}=\mathrm{y}\\
0 &  \hspace*{0.1in}\mbox{if}\hspace*{0.1in}\mathrm{x}\neq\mathrm{y}
\end{array}\right.
\end{equation}
for $a^\mathrm{x}\in A^\mathrm{x}$, $b^\mathrm{y}\in A^\mathrm{y}$ and 
\begin{align}
\label{product_def1}
\m{a}{b}:=\sum_{\mathrm{x},\mathrm{y}\in \C{G}_0} \m{a^\mathrm{x}}{b^\mathrm{y}}=\sum_{\mathrm{x}\in \C{G}_0} \m{a^\mathrm{x}}{b^\mathrm{x}}
\end{align}
for $a=\sum_{\mathrm{x}\in \C{G}_0}a^\mathrm{x}$, $b=\sum_{\mathrm{x}\in \C{G}_0}b^\mathrm{x}$ where $a^\mathrm{x}$, $b^\mathrm{x}\in A^\mathrm{x}$ for all $\mathrm{x}\in \C{G}_0$. Then
\begin{itemize}
\item[(i)] for $a^\mathrm{x}_g\in A^\mathrm{x}_g$ and $b^\mathrm{x}_h\in A^\mathrm{x}_h$, 
\begin{equation}
\label{GalgebraProduct}
\m{a^\mathrm{x}_g}{b^\mathrm{x}_h}\in A^\mathrm{x}_{gh}.
\end{equation}
In particular, $\m{a^\mathrm{x}}{b^\mathrm{y}}=\delta_{\mathrm{x},\mathrm{y}}~\m{a^\mathrm{x}}{b^\mathrm{y}}\in A^\mathrm{x}$;\vspace*{0.05in}
\item[(ii)]  $A$ is a unital associative algebra over $k$ with multiplication $\m{}{}$  and unit \\$\textbf{1}_A:=\mu(1)=\sum_{\mathrm{x}\in \C{G}_0}\mu(1^\mathrm{x})$;\vspace*{0.05in}
\item[(iii)]  if $(A,m,\mu)$ is also commutative, then 
\begin{equation}
\m{a^\mathrm{x}_g}{b^\mathrm{x}_h}=\m{(\varphi(g)b^\mathrm{x}_h)}{a^\mathrm{x}_g}
\end{equation}
where $\varphi$ is the $\C{G}$-action given by Corollary \ref{GroupoidActionCor}.
\end{itemize}
 \end{proposition}
\begin{proof}
For (\ref{GalgebraProduct}), we have
\begin{align}
\m{a^\mathrm{x}_g}{b^\mathrm{x}_h}&=m(a^\mathrm{x}_g\otimes b^\mathrm{x}_h)\\
\label{GalgebraProduct1}
&=m(\p{g}{e_\mathrm{x}}\rhd a^\mathrm{x}_g\otimes \p{h}{e_\mathrm{x}}\rhd b^\mathrm{x}_h)\\
\label{GalgebraProduct2}
&=m\left(\sum_{\{g_1,h_1\in \Gamma^\mathrm{x}~|~g_1h_1=gh\}}\p{g_1}{e_\mathrm{x}}\rhd a^\mathrm{x}_g\otimes \p{h_1}{e_\mathrm{x}}\rhd b^\mathrm{x}_h\right)\\
\label{GalgebraProduct3}
&=\p{gh}{e_\mathrm{x}}\rhd m(a^\mathrm{x}_g\otimes b^\mathrm{x}_h)\\
\label{GalgebraProduct4}
&=\p{gh}{e_\mathrm{x}}\rhd (\m{a^\mathrm{x}_g}{b^\mathrm{x}_h})\in A^\mathrm{x}_{gh}
\end{align}
where (\ref{GalgebraProduct1}), (\ref{GalgebraProduct2}), and (\ref{GalgebraProduct4}) follow from statement (ii) of Proposition \ref{DirectSum1} and (\ref{GalgebraProduct3}) follows from the fact that 
\begin{itemize}
 \item[(1)] $m$ is a $D(k[\C{G}])$-linear map from $A\widehat{\otimes}A$ to $A$ where $A\widehat{\otimes}A=\bigoplus_{\mathrm{y}\in \C{G}_0}A^\mathrm{y}\otimes_k A^\mathrm{y}$ by Lemma \ref{tensor_product2a};
 \item[(2)] the $D(k[\C{G}])$-action on $A\widehat{\otimes}A$ is induced by the coproduct of $D(k[\C{G}])$; and
 \item[(3)] $\Delta(\p{gh}{e_\mathrm{x}})=\displaystyle\sum_{\{g_1,h_1\in \Gamma^\mathrm{x}~|~g_1h_1=gh\}}\p{g_1}{e_\mathrm{x}}\otimes  \p{h_1}{e_\mathrm{x}}$.
 \end{itemize}
 Since $A^\mathrm{x}=\bigoplus_{g\in \Gamma^\mathrm{x}} A^\mathrm{x}_g$, it follows readily from $(\ref{GalgebraProduct})$ (and the definition of the product) that $\m{a^\mathrm{x}}{b^\mathrm{y}}=\delta_{\mathrm{x},\mathrm{y}}~\m{a^\mathrm{x}}{b^\mathrm{y}}\in A^\mathrm{x}$ for $a^\mathrm{x}\in A^\mathrm{x}$, $b^\mathrm{y}\in A^\mathrm{y}$.  This completes the proof of (i).
 
For (ii), note that since $m$ is also $k$-linear, it follows that $\m{(\lambda_1 a)}{(\lambda_2 b)}= (\lambda_1\lambda_2)(\m{a}{b})$ for $\lambda_1,\lambda_2\in k$ and $a,b\in A$.   To show that $\textbf{1}_A$ is the unit element of $A$, we use the fact that $m\circ (\mu\otimes id_A)=l_A$ and $m\circ (id_A\otimes \mu)=r_A$ (where $l_A$ and $r_A$ denotes the left and right identity maps of $A$).  The latter shows that 
 \begin{align}
 \label{LeftMultiplication}
 m(\mu(1^\mathrm{x})\otimes a^\mathrm{x})&=l_A(1^\mathrm{x}\otimes a^\mathrm{x})=a^\mathrm{x}\\
 \label{RightMultiplication}
 m(a^\mathrm{x}\otimes \mu(1^\mathrm{x}))&=r_A(a^\mathrm{x}\otimes 1^\mathrm{x})=a^\mathrm{x}
 \end{align}
 for $a^\mathrm{x}\in A^\mathrm{x}$, where part (iii) of Lemma \ref{UnitObject} has been applied in (\ref{LeftMultiplication}) and (\ref{RightMultiplication}).    Now for $a\in A$, $a$ can be uniquely written as $\sum_{\mathrm{x}\in \C{G}_0}a^\mathrm{x}$ where $a^\mathrm{x}\in A^\mathrm{x}$.   By (\ref{LeftMultiplication}) and (\ref{RightMultiplication}), we have the following:
 \begin{align}
 \m{\textbf{1}_A}{a}&=\sum_{\mathrm{x}\in \C{G}_0}\m{\mu(1^\mathrm{x})}{a^\mathrm{x}}=\sum_{\mathrm{x}\in \C{G}_0}m(\mu(1^\mathrm{x})\otimes a^\mathrm{x})=\sum_{\mathrm{x}\in \C{G}_0}a^\mathrm{x}=a\\
 \m{a}{\textbf{1}_A}&=\sum_{\mathrm{x}\in \C{G}_0}\m{a^\mathrm{x}}{\mu(1^\mathrm{x})}=\sum_{\mathrm{x}\in \C{G}_0}m(a^\mathrm{x}\otimes \mu(1^\mathrm{x}))=\sum_{\mathrm{x}\in \C{G}_0}a^\mathrm{x}=a.
 \end{align}
 
 For associativity, it suffices to check that
\begin{equation}
\label{assoc1}
\m{(\m{a^\mathrm{x}}{b^\mathrm{y}})}{c^\mathrm{z}}=\m{a^\mathrm{x}}{(\m{b^\mathrm{y}}{c^\mathrm{z}})}
\end{equation}
for $a^\mathrm{x}\in A^\mathrm{x}$, $b^\mathrm{y}\in A^\mathrm{y}$, and $c^\mathrm{z}\in A^\mathrm{z}$.  From the definition of the product, its easy to see that both sides of (\ref{assoc1}) are zero when $\mathrm{x}\neq \mathrm{y}$ or $\mathrm{y}\neq \mathrm{z}$.  For the case when $\mathrm{x}=\mathrm{y}=\mathrm{z}$, we have
\begin{equation}
\label{assoc2}
\m{(\m{a^\mathrm{x}}{b^\mathrm{x}})}{c^\mathrm{x}}=m(m(a^\mathrm{x}\otimes b^\mathrm{x})\otimes c^\mathrm{x})
\end{equation}
and 
\begin{equation}
\label{assoc3}
\m{a^\mathrm{x}}{(\m{b^\mathrm{x}}{c^\mathrm{x}})}=m(a^\mathrm{x}\otimes m(b^\mathrm{x}\otimes c^\mathrm{x})).
\end{equation}
Since  
\begin{equation}
m\circ (m\otimes id_A)=m\circ(id_A\otimes m)
\end{equation}
we see that (\ref{assoc2}) and (\ref{assoc3}) are indeed equal.  In addition, it follows readily from (\ref{product_def}) and (\ref{product_def1}) that 
\begin{align}
\m{a}{(b+c)}&=\m{a}{b}+\m{a}{c}\\
\m{(b+c)}{a}&=\m{b}{a}+\m{c}{a}
\end{align}
for $a,b,c\in A$.  This completes the proof of (ii).  

Lastly for (iii), note that if $(A,m,\mu)$ is also commutative, we have
\begin{align}
\m{a^\mathrm{x}_g}{b^\mathrm{x}_h}&=m(a^\mathrm{x}_g\otimes b^\mathrm{x}_h)\\
&=m\circ c_{A,A}(a^\mathrm{x}_g\otimes b^\mathrm{x}_h)\\
&=m \left(\sum_{\mathrm{y}\in \C{G}_0} \sum_{l,m\in \Gamma^\mathrm{y}} (\p{m}{l}\rhd b^\mathrm{x}_h)\otimes (\p{l}{e_\mathrm{y}}\rhd a^\mathrm{x}_g)\right)\\
&=m\left((\p{ghg^{-1}}{g}\rhd b^\mathrm{x}_h)\otimes (\p{g}{e_\mathrm{x}}\rhd a^\mathrm{x}_g)\right)\\
&=m\left((\p{ghg^{-1}}{g}\rhd b^\mathrm{x}_h)\otimes a^\mathrm{x}_g\right)\\
&=\m{(\p{ghg^{-1}}{g}\rhd b^\mathrm{x}_{h})}{a^\mathrm{x}_g}
\end{align}
where the second equality follows from the fact that $m=m\circ c_{A,A}$ (since $(A,m,\mu)$ is a \textit{commutative} algebra object); the third equality follows from the definition of the braiding morphism $c$, which is given by (\ref{BraidingDef}), and the definition of the R-matrix of $D(k[\C{G}])$, which is given by (\ref{RmatrixDkG}) and (\ref{Rx}); and the fourth and fifth equalities follow from parts (ii) and (iii) of Proposition \ref{DirectSum1}.  Since
\begin{equation}
\p{ghg^{-1}}{g}\rhd b^\mathrm{x}_h =\sum_{l\in \Gamma^\mathrm{x}} \p{l}{g}\rhd b^\mathrm{x}_h
\end{equation}
(by (iii-a) of Proposition \ref{DirectSum1}) and the right side is just $\varphi(g)b^\mathrm{x}_h$, we have 
\begin{equation}
\m{a^\mathrm{x}_g}{b^\mathrm{x}_h}=\m{(\varphi(g)b^\mathrm{x}_h)}{a^\mathrm{x}_g}
\end{equation}
and this completes the proof of Proposition \ref{ProductResult1}.
\end{proof}
\noindent We now show how any left $D(k[\C{G}])$-module $(\rho,A)$ which is both an algebra and colagebra object gives rise to a bilinear form which satisfies all the axioms of a $\C{G}$-FA except possibly axiom (vi) of Definition \ref{DefGrpdFA}.  We will see later in Proposition \ref{etaNonDegen} that to ensure that the induced bilinear form is nondegenerate (i.e., satisfies axiom (vi) of Definition \ref{DefGrpdFA}), the algebra and coalgebra structure on $(\rho,A)$ must satisfy the Frobenius relations (equations (\ref{FrobeniusAxiom1}) and (\ref{FrobeniusAxiom2})).  In other words, $(\rho,A)$ must also be a Frobenius object.   We begin by examining the properties of a coalgebra object in $\mbox{Rep}(D(k[\C{G}]))$.
\begin{lemma}
\label{CounitRepDkG}
Let $(A,\Delta,\varepsilon)$ be a coalgebra object in $\mbox{Rep}(D(k[\C{G}]))$.  Then \\$\varepsilon: A\rightarrow D(k[\C{G}])_t$ and $\Delta: A\rightarrow A\widehat{\otimes} A$ satisfy the following:
\begin{itemize}
\item[(i)] $\varepsilon(a^\mathrm{x}_g)=\delta_{g,e_\mathrm{x}}\varepsilon(a^\mathrm{x}_g)\in k1^\mathrm{x}\subset D(k[\C{G}])_t$, and
\item[(ii)] $\Delta(a^\mathrm{x}_g)\in\displaystyle \bigoplus_{\{g_1,g_2\in \Gamma^\mathrm{x}~|~g_1g_2=g\}}A^\mathrm{x}_{g_1}\otimes_k A^\mathrm{x}_{g_2}$
\end{itemize}
for $a^\mathrm{x}_g\in A^\mathrm{x}_g$. 
\end{lemma}
\begin{proof}
By part (i-a) of Lemma \ref{UnitObject}, $\varepsilon(a^\mathrm{x}_g)$ can be written as
\begin{equation}
\varepsilon(a^\mathrm{x}_g)= \sum_{\mathrm{y}\in \C{G}_0} \lambda^\mathrm{y} 1^\mathrm{y},~~\lambda^\mathrm{y}\in k.
\end{equation}
Then 
\begin{align}
\varepsilon(a^\mathrm{x}_g)&=\varepsilon(\p{g}{e^\mathrm{x}}\rhd a^\mathrm{x}_g)\\
&=\p{g}{e_\mathrm{x}}\rhd\varepsilon(a^\mathrm{x}_g)\\
&=\sum_{\mathrm{y}\in \C{G}_0} \lambda^\mathrm{y}~\p{g}{e_\mathrm{x}}\rhd 1^\mathrm{y}\\
&=\sum_{\mathrm{y}\in \C{G}_0} \delta_{\mathrm{x},\mathrm{y}}\delta_{g,e_\mathrm{x}} \lambda^\mathrm{y}~1^\mathrm{y}\\
&=\delta_{g,e_\mathrm{x}} \lambda^\mathrm{x}1^\mathrm{x}
\end{align}
which proves part (i) of Lemma \ref{CounitRepDkG}.  In the above calculation, the fourth equality follows from part (i-b) of Lemma \ref{UnitObject}.

For part (ii), we have
\begin{align}
\Delta(a^\mathrm{x}_g)&=\Delta(\p{g}{e_\mathrm{x}}\rhd a^\mathrm{x}_g)\\
\label{secondline}
&=\sum_{g_1g_2=g}\left(\p{g_1}{e_\mathrm{x}}\rhd (a^\mathrm{x}_g)_{(1)}\right)\otimes \left(\p{g_2}{e_\mathrm{x}}\rhd (a^\mathrm{x}_g)_{(2)}\right)\in \bigoplus_{g_1g_2=g}A^\mathrm{x}_{g_1}\otimes_k A^\mathrm{x}_{g_2}
\end{align}
where the above calculation follows from part (ii) of Proposition \ref{DirectSum1} and the fact that $\p{g}{e_\mathrm{x}}$ acts on $\Delta(a^\mathrm{x}_g)=(a^\mathrm{x}_g)_{(1)}\otimes (a^\mathrm{x}_g)_{(2)}$ via the coproduct of $D(k[\C{G}])$.  (Note that the sum and direct sum in (\ref{secondline}) are over all $g_1,g_2\in \Gamma^\mathrm{x}$ such that $g_1g_2=g$.)
\end{proof}
\begin{proposition}
\label{etaResult1}
Suppose $(A,m,\mu)$ and $(A,\Delta,\varepsilon)$ are respectively algebra and colagebra objects in $Rep(D(k[\C{G}]))$.  Let $\varepsilon': A\rightarrow k$ be the $k$-linear map defined by\footnote{Note that by part (i) of Lemma \ref{CounitRepDkG} the coefficient of $1^\mathrm{x}$ in (\ref{counitprime}) depends only upon $a^\mathrm{x}$, the $\mathrm{x}$-component of $a$.}
\begin{equation}
\label{counitprime}
\varepsilon(a)=\sum_{\mathrm{x}\in \C{G}_0} \varepsilon'(a^\mathrm{x})~1^\mathrm{x}
\end{equation}
for $a=\sum_{\mathrm{x}\in \C{G}_0}a^\mathrm{x}$ with $a^\mathrm{x}\in A^\mathrm{x}$ and let $\eta: A\times A\rightarrow k$ be the map defined by
\begin{equation}
\label{eta_def1}
\eta(a,b)=\varepsilon'(\m{a}{b})
\end{equation}
where $\m{a}{b}$ is the product given in Proposition \ref{ProductResult1}.  Then 
\begin{itemize}
\item[(i)] $\eta$ is a $k$-bilinear map;
\item[(ii)] $\eta(\m{a}{b},c)=\eta(a,\m{b}{c})$ for $a,b,c\in A$
\item[(iii)] $\eta(a^\mathrm{x}_g,b^\mathrm{x}_h)=0$ for $gh\neq e_\mathrm{x}$ where $a^\mathrm{x}_g\in A^\mathrm{x}_g$ and $b^\mathrm{x}_h\in A^\mathrm{x}_h$; and
\item[(iv)] $\eta(a^\mathrm{x},b^\mathrm{x})=\eta(\varphi(x)a^\mathrm{x},\varphi(x)b^\mathrm{x})$ where $\varphi$ is the $\C{G}$-action defined in Corollary \ref{GroupoidActionCor}, $a^\mathrm{x},b^\mathrm{x}\in A^\mathrm{x}$, and $x\in \C{G}_1$ with $s(x)=\mathrm{x}$.
\end{itemize}
\end{proposition}
\begin{proof}
(i) follows easily from the definition of $\eta$ and (ii) follows from the fact that $\m{\m{(a}{b)}}{c}=\m{a}{\m{(b}{c)}}$ by Proposition \ref{ProductResult1}.  

For (iii), note that since $\m{a^\mathrm{x}_g}{b^\mathrm{x}_h}\in A^\mathrm{x}_{gh}$ (by Proposition \ref{ProductResult1}), we have 
\begin{equation}
\label{counitprime2}
\varepsilon(\m{a^\mathrm{x}_g}{b^\mathrm{x}_h})=\varepsilon'(\m{a^\mathrm{x}_g}{b^\mathrm{x}_h})1^\mathrm{x}.
\end{equation}
With $\m{a^\mathrm{x}_g}{b^\mathrm{x}_h}\in A^\mathrm{x}_{gh}$, it follows from part (i) of Lemma \ref{CounitRepDkG} that $\varepsilon(\m{a^\mathrm{x}_g}{b^\mathrm{x}_h})=0$ for $gh\neq e_\mathrm{x}$.  By  (\ref{counitprime2}), $\varepsilon'(\m{a^\mathrm{x}_g}{b^\mathrm{x}_h})$ is also zero for $gh\neq e_\mathrm{x}$ and this proves (iii).

For (iv), it suffices to consider the case when $a^\mathrm{x}=a^\mathrm{x}_g\in A^\mathrm{x}_g$ and $b^\mathrm{x}=b^\mathrm{x}_h\in A^\mathrm{x}_h$.   By Corollary \ref{GroupoidActionCor}, $\varphi(x)a^\mathrm{x}_g\in A^{\mathrm{y}}_{xgx^{-1}}$ and $\varphi(x)b^\mathrm{x}_h\in A^{\mathrm{y}}_{xhx^{-1}}$ where we have set $\mathrm{y}=t(x)$.   If $gh\neq e_\mathrm{x}$, then 
\begin{align}
\eta(a^\mathrm{x}_g,b^\mathrm{x}_h)=\eta(\varphi(x)a^\mathrm{x}_g,\varphi(x)b^\mathrm{x}_h)=0
\end{align}
by part (iii) of Proposition \ref{etaResult1}.

For the case when $gh=e_\mathrm{x}$, we have 
\begin{align}
\nonumber
\eta(\varphi(x)a^\mathrm{x}_g,\varphi(x)b^\mathrm{x}_{g^{-1}})&=\varepsilon'\left(\m{(\varphi(x)a^\mathrm{x}_g)}{(\varphi(x)b^\mathrm{x}_{g^{-1}})}\right)\\
\nonumber
&=\varepsilon'\left(\m{(\p{xgx^{-1}}{x}\rhd a^\mathrm{x}_g)}{(\p{xg^{-1}x^{-1}}{x}\rhd b^\mathrm{x}_{g^{-1}})}\right)\\
\nonumber
&=\varepsilon'\left(m((\p{xgx^{-1}}{x}\rhd a^\mathrm{x}_g)\otimes (\p{xg^{-1}x^{-1}}{x}\rhd b^\mathrm{x}_{g^{-1}}))\right)\\
\nonumber
&=\varepsilon'\left(m\left(\sum_{uv=e_\mathrm{y},~u,v\in \Gamma^\mathrm{y}}(\p{u}{x}\rhd a^\mathrm{x}_g)\otimes (\p{v}{x}\rhd b^\mathrm{x}_{g^{-1}})\right)\right)\\
\nonumber
&=\varepsilon'\left(\p{e_\mathrm{y}}{x}\rhd m(a^\mathrm{x}_g\otimes b^\mathrm{x}_{g^{-1}})\right)\\
\label{counitprime3}
&=\varepsilon'\left(\p{e_\mathrm{y}}{x}\rhd(\m{a^\mathrm{x}_g}{b^\mathrm{x}_{g^{-1}}})\right)
\end{align}
where the second equality follows from the definition of $\varphi$ and part (iii-a) of Proposition \ref{DirectSum1}; the fourth equality also follows from part (iii-a) of Proposition \ref{DirectSum1}; and the fifth equality follows from the fact that $m$ is $D(k[\C{G}])$-linear.   In addition, 
\begin{align}
\varepsilon\left(\p{e_\mathrm{y}}{x}\rhd(\m{a^\mathrm{x}_g}{b^\mathrm{x}_{g^{-1}}})\right)&=\p{e_\mathrm{y}}{x}\rhd\varepsilon\left(\m{a^\mathrm{x}_g}{b^\mathrm{x}_{g^{-1}}}\right)\\
&=\varepsilon'\left(\m{a^\mathrm{x}_g}{b^\mathrm{x}_{g^{-1}}}\right)~\p{e_\mathrm{y}}{x}\rhd1^\mathrm{x}\\
&=\varepsilon'\left(\m{a^\mathrm{x}_g}{b^\mathrm{x}_{g^{-1}}}\right)~1^\mathrm{y}
\end{align}
where the third equality follows from part (i-b) of Lemma \ref{UnitObject}.  

Since $\p{e_\mathrm{y}}{x}\rhd(\m{a^\mathrm{x}_g}{b^\mathrm{x}_{g^{-1}}})\in A^\mathrm{y}_{e_\mathrm{y}}$, we also have
\begin{equation}
\varepsilon\left(\p{e_\mathrm{y}}{x}\rhd(\m{a^\mathrm{x}_g}{b^\mathrm{x}_{g^{-1}}})\right)=\varepsilon'\left(\p{e_\mathrm{y}}{x}\rhd(\m{a^\mathrm{x}_g}{b^\mathrm{x}_{g^{-1}}})\right)~1^\mathrm{y}.
\end{equation}
Hence, 
\begin{equation}
\varepsilon'\left(\p{e_\mathrm{y}}{x}\rhd(\m{a^\mathrm{x}_g}{b^\mathrm{x}_{g^{-1}}})\right)
=\varepsilon'\left(\m{a^\mathrm{x}_g}{b^\mathrm{x}_{g^{-1}}}\right).
\end{equation}
It follows from this as well as the definition of $\eta$ and (\ref{counitprime3}) that
\begin{equation}
\eta(\varphi(x)a^\mathrm{x}_g,\varphi(x)b^\mathrm{x}_{g^{-1}})=\eta(a^\mathrm{x}_g,b^\mathrm{x}_{g^{-1}})
\end{equation}
and this completes the proof of (iv).
\end{proof}
\begin{proposition}
\label{etaNonDegen}
Let $(A,m,\triangle,\mu,\varepsilon)$ be a Frobenius object in $\mbox{Rep}(D(k[\C{G}]))$ and let $\eta: A\times A\rightarrow k$ be the bilinear form given by Proposition \ref{etaResult1}.  Then $\eta|_{A^\mathrm{x}_g\times A^\mathrm{x}_h}$ is nondegenerate for all $\mathrm{x}\in \C{G}_0$ and $g,h\in \Gamma^\mathrm{x}$ satisfying $gh=e_\mathrm{x}$.
\end{proposition}
\begin{proof}
To start, set
\begin{equation}
1^\mathrm{x}_A:=\mu(1^\mathrm{x})\in A^\mathrm{x}.
\end{equation}
Then from the proof of Proposition \ref{ProductResult1}, we have  
\begin{equation}
\m{1^\mathrm{x}_A}{a^\mathrm{x}}=\m{a^\mathrm{x}}{1^\mathrm{x}_A}=a^\mathrm{x}
\end{equation}
for $a^\mathrm{x}\in A^\mathrm{x}$.   Furthermore, (\ref{GalgebraProduct}) of Proposition \ref{ProductResult1} implies that $1^\mathrm{x}_A\in A^\mathrm{x}_{e_\mathrm{x}}$.  By (ii) of Lemma \ref{CounitRepDkG},
\begin{equation}
\Delta(1^\mathrm{x}_A)\in \bigoplus_{h\in \Gamma^\mathrm{x}} A^\mathrm{x}_h\otimes_k A^\mathrm{x}_{h^{-1}}.
\end{equation}
Now let $v_1,\dots,v_n$ be a basis for $A^\mathrm{x}$ with $v_i\in A^{\mathrm{x}}_{h_i}$ for some $h_i\in \Gamma^\mathrm{x}$.  Then 
\begin{equation}
\Delta(1^\mathrm{x}_A) = \sum_{i=1}^n u_i\otimes v_i
\end{equation}   
for some $u_i\in A^\mathrm{x}_{h_i^{-1}}$.  

If $a^\mathrm{x}\in A^\mathrm{x}$, then  
\begin{align}
a^\mathrm{x}&=l_A\circ (\varepsilon\otimes id_A)\circ \Delta(a^\mathrm{x})\\
&= l_A\circ (\varepsilon\otimes id_A)\circ \Delta(\m{a^\mathrm{x}}{1^\mathrm{x}_A})\\
&= l_A\circ (\varepsilon\otimes id_A)\circ \Delta\circ m(a^\mathrm{x}\otimes 1^\mathrm{x}_A)\\
&=l_A\circ (\varepsilon\otimes id_A)\circ (m\otimes id_A)\circ (id_A\otimes \Delta)(a^\mathrm{x}\otimes 1^\mathrm{x}_A)\\
&=\sum_{i=1}^n \varepsilon'(\m{a^\mathrm{x}}{u_i})~v_i,
\end{align}
where the first equality is just the counit property of a coalgebra object; the fourth equality follows from (\ref{FrobeniusAxiom1}); and the fifth equality employs the linear map $\varepsilon':A\rightarrow k$ that was defined in Proposition \ref{etaResult1}.   Setting $a^\mathrm{x}=v_j$ and using the fact that the $v_i$'s are linearly independent gives
\begin{equation}
\label{vuResult}
\eta(v_j,u_i):= \varepsilon'(\m{v_j}{u_i})=\delta_{ji}.
\end{equation}
Using (\ref{FrobeniusAxiom2}), a similar calculation shows that
\begin{equation}
\label{uLinDep}
a^\mathrm{x}=\sum_{i=1}^n\varepsilon'(\m{v_i}{a^\mathrm{x}})~u_i.
\end{equation}
Since $a^\mathrm{x}$ is arbitrary and the dimension of $A^\mathrm{x}$ is $n$, (\ref{uLinDep}) shows that $\{u_i\}_{i=1}^n$ is also a basis of $A^\mathrm{x}$. 

 (\ref{vuResult}) combined with the fact that $\{v_i\}_{i=1}^n$ and $\{u_i\}_{i=1}^n$ are both bases of $A^\mathrm{x}$ (where $v_i\in A^\mathrm{x}_{h_i}$ and $u_i\in A^\mathrm{x}_{h_i^{-1}}$) shows that $\eta|_{A^\mathrm{x}_g\times A^\mathrm{x}_h}$ is nondegenerate for $gh=e_\mathrm{x}$.  
\end{proof}
\noindent  The next result establishes the second half of Theorem \ref{MainTheorem}.
\begin{proposition}
\label{GFO2GFA}
If $((A,\rho),m,\Delta,\mu,\varepsilon)$ is a Frobenius object in $\mbox{Rep}(D(k[\C{G}]))$ which satisfies conditions (1) and (2) of Theorem \ref{MainTheorem}, then
\begin{equation}
<\C{G},(A,\m{}{},\textbf{1}_A),\eta,\varphi>
\end{equation}
is a $\C{G}$-FA where 
\begin{itemize}
\item[(i)] $\m{}{}$ and $\textbf{1}_A$ are respectively the product and multiplicative unit given by Proposition \ref{ProductResult1};
\item[(ii)] $\eta$ is the bilinear form given by Proposition \ref{etaResult1}; and
\item[(iii)] $\varphi$ is the $\C{G}$-action given by Corollary \ref{GroupoidActionCor}.
\end{itemize}
\end{proposition}
\begin{proof}
Axioms (b), (i), and (ii) of Definition \ref{DefGrpdFA} are satisfied by parts (i) and (ii) of Proposition \ref{DirectSum1} and parts (i) and (ii) of Proposition \ref{ProductResult1}.

Axioms (v) and (vii) of Definition \ref{DefGrpdFA} are satisfied by Corollary \ref{GroupoidActionCor} and  part (iii) of Proposition \ref{ProductResult1} respectively.  

For axiom (d), we only need to verify that $\varphi(x): A^\mathrm{x}\rightarrow A^\mathrm{y}$ is an algebra isomorphism for $x\in \C{G}_1$ where $s(x)=\mathrm{x}$ and $t(x)=\mathrm{y}$.  By Corollary \ref{GroupoidActionCor}, $\varphi(x)$ is already an isomorphism of vector spaces.  Hence, we only need to check that 
\begin{equation}
\varphi(x)(\m{a^\mathrm{x}}{b^\mathrm{x}})=\m{(\varphi(x)a^\mathrm{x})}{(\varphi(x)b^\mathrm{x})}.
\end{equation} 
It suffices to verify this for the case when $a^\mathrm{x}=a^\mathrm{x}_g$ and $b^\mathrm{x}=b^\mathrm{x}_h$.    In this case, we have  
\begin{align}
\varphi(x)(\m{a^\mathrm{x}_g}{b^\mathrm{x}_h})&=\p{xghx^{-1}}{x}\rhd (\m{a^\mathrm{x}_g}{b^\mathrm{x}_h})\\
&=\p{xghx^{-1}}{x}\rhd m({a^\mathrm{x}_g}\otimes{b^\mathrm{x}_h})\\
&= m\left((\p{xgx^{-1}}{x}\rhd{a^\mathrm{x}_g})\otimes(\p{xhx^{-1}}{x}\rhd{b^\mathrm{x}_h})\right)\\
&= m\left((\varphi(x){a^\mathrm{x}_g})\otimes(\varphi(x){b^\mathrm{x}_h})\right)\\
&=\m{(\varphi(x){a^\mathrm{x}_g})}{(\varphi(x){b^\mathrm{x}_h})}.
\end{align}
Throughout the above calculation we have made use of part (iii) of Proposition \ref{DirectSum1}, and in the third equality, we have made use of the fact $m$ is $D(k[\C{G}])$-linear.

Axioms (c), (iii), and (iv) of Definition \ref{DefGrpdFA} are satisfied by Proposition \ref{etaResult1}; axiom (vi) of Definition \ref{DefGrpdFA} is satisfied by part (iii) of Proposition \ref{etaResult1} and by Proposition \ref{etaNonDegen}.

All that remains left to do is to show that axioms (viii) and (ix) of  Definition \ref{DefGrpdFA} are also satisfied.  We will now show that axioms (viii) and(ix) follow respectively from conditions (1) and (2) of Theorem \ref{MainTheorem}.  

For axiom (viii) of Definition \ref{DefGrpdFA}, let $a^\mathrm{x}_g\in A^\mathrm{x}_g$.  Then
\begin{align}
a^\mathrm{x}_g&=\sum_{\mathrm{y}\in \C{G}_0}\sum_{h\in \Gamma^\mathrm{y}} \p{h}{h}\rhd a^\mathrm{x}_g\\
&=\p{g}{g}\rhd a^\mathrm{x}_g\\
&=\varphi(g)a^\mathrm{x}_g
\end{align}
where the first equality follows from condition (1) of Theorem \ref{MainTheorem} and the second and third equalities follow from part (iii-a) of Proposition \ref{DirectSum1}.  This shows that axiom (viii) of  Definition \ref{DefGrpdFA} is satisfied.

For axiom (ix) of Definition \ref{DefGrpdFA}, let $g,h\in A^\mathrm{x}$ and for $c\in A^\mathrm{x}_{ghg^{-1}h^{-1}}$, let $l_c:A^\mathrm{x}\rightarrow A^\mathrm{x}$ be the linear map defined by $l_c(a^\mathrm{x}):=\m{c}{a^\mathrm{x}}$.  Then by part (iii) of Proposition \ref{DirectSum1} and by part (i) of Proposition \ref{ProductResult1}, we have the following:
\begin{align}
\mbox{Tr}\left(l_c\circ \rho(\p{hgh^{-1}}{h})\right)&=\mbox{Tr}\left(l_c\circ \rho(\p{hgh^{-1}}{h})|_{A^\mathrm{x}_g}:A^\mathrm{x}_g\rightarrow A^\mathrm{x}_g\right)\\
\mbox{Tr}\left(\rho(\p{h}{g^{-1}})\circ l_c\circ \rho(\p{h}{e_\mathrm{x}})\right)&= \mbox{Tr}\left(\rho(\p{h}{g^{-1}})\circ l_c|_{A^\mathrm{x}_h}: A^\mathrm{x}_h\rightarrow A^\mathrm{x}_h\right).
\end{align}  
Condition (2) of Theorem \ref{MainTheorem} then gives
 \begin{align}
 \label{axiom9A}
 \mbox{Tr}\left(l_c\circ \rho(\p{hgh^{-1}}{h})|_{A^\mathrm{x}_g}:A^\mathrm{x}_g\rightarrow A^\mathrm{x}_g\right)
=\mbox{Tr}\left(\rho(\p{h}{g^{-1}})\circ l_c|_{A^\mathrm{x}_h}: A^\mathrm{x}_h\rightarrow A^\mathrm{x}_h\right).
\end{align}
Since
 \begin{align}
  l_c\circ \rho(\p{hgh^{-1}}{h})|_{A^\mathrm{x}_g}=l_c\circ \varphi(h)|_{A^\mathrm{x}_g}
  \end{align}
  and
  \begin{align}
  \rho(\p{h}{g^{-1}})\circ l_c|_{A^\mathrm{x}_h}=  \varphi(g^{-1})\circ l_c|_{A^\mathrm{x}_h}
 \end{align}
 by part (iii) of Proposition \ref{DirectSum1} and the definition of $\varphi$, (\ref{axiom9A}) shows that axiom (ix) is satisfied and this completes the proof of Proposition \ref{GFO2GFA}.
\end{proof}
\subsection{Frobenius objects via $\C{G}$-FAs}
In this section, we move in the opposite direction and show that every $\C{G}$-FA is also a Frobenius object in $\mbox{Rep}(D(k[\C{G}]))$ which satisfies conditions (1) and (2) of Theorem \ref{MainTheorem}.  We begin with the following result:
\begin{proposition}
\label{InducedDkGResult}
Every $\C{G}$-FA is a left $D(k[\C{G}])$-module.  If $A$ is a $\C{G}$-FA with $\C{G}$-action $\varphi$, then its $D(k[\C{G}])$-action is the linear map defined by
\begin{equation}
\label{InducedDkGAction}
\rho(\p{g}{x})a^\mathrm{y}_h:=\delta_{h,x^{-1}gx}~\varphi(x)a^\mathrm{y}_h
\end{equation}
for $\p{g}{x}\in D(k[\C{G}])$ and $a^\mathrm{y}_h\in A^\mathrm{y}_h$.
\end{proposition}
\begin{proof}
To start, let $a\in A$ and decompose it as $a=\sum_{\mathrm{x}\in\C{G}_0}\sum_{g\in \Gamma^\mathrm{x}}a^\mathrm{x}_g$.  To show that (\ref{InducedDkGAction}) does indeed define a $D(k[\C{G}])$-action, we need to verify that 
 \begin{itemize}
 \item[(i)] $\rho(1)=id_A$, and 
 \item[(ii)] $\rho(\p{g_1}{x_1})\circ \rho(\p{g_2}{x_2})=\rho(\p{g_1}{x_1}\cdot \p{g_2}{x_2})$. 
 \end{itemize}
For (i), we have
\begin{align}
\nonumber
\rho(1)a=\sum_{\mathrm{x}\in \C{G}_0} \sum_{g\in \Gamma^\mathrm{x}}\rho(\p{g}{e_\mathrm{x}})a
=\sum_{\mathrm{x}\in \C{G}_0} \sum_{g\in \Gamma^\mathrm{x}}\varphi(e_\mathrm{x})a^\mathrm{x}_{g}
=\sum_{\mathrm{x}\in \C{G}_0} \sum_{g\in \Gamma^\mathrm{x}}a^\mathrm{x}_{g}
=a.
\end{align}
For (ii), we have 
\begin{align}
\nonumber
\rho(\p{g_1}{x_1})\circ \rho(\p{g_2}{x_2})a&=\rho(\p{g_1}{x_1})\circ\varphi(x_2)a^{s(x_2)}_{x_2^{-1}g_2x_2}\\
\label{InducedDkGAction1}
&=\delta_{x_1^{-1}g_1x_1,g_2}\varphi(x_1)\circ \varphi(x_2)a^{s(x_2)}_{x_2^{-1}g_2x_2}.
\end{align}
Since $\p{g_1}{x_1}\cdot \p{g_2}{x_2}=\delta_{x_1^{-1}g_1x_1,g_2}\p{g_1}{x_1x_2}$, we see that (ii) is satisfied for the case when $x_1^{-1}g_1x_1\neq g_2$.  For the case when $x_1^{-1}g_1x_1= g_2$, (\ref{InducedDkGAction1}) reduces to
\begin{align}
\varphi(x_1x_2)a^{s(x_2)}_{x_2^{-1}g_2x_2}&=\rho(\p{x_1g_2x_1^{-1}}{x_1x_2})a^{s(x_2)}_{x_2^{-1}g_2x_2}\\
&=\rho(\p{g_1}{x_1x_2})a^{s(x_2)}_{x_2^{-1}g_2x_2}\\
&=\rho(\p{g_1}{x_1x_2})a\\
&=\rho(\p{g_1}{x_1}\cdot \p{g_2}{x_2})a
\end{align}
and this completes the proof of (ii).
\end{proof}
\noindent Proposition \ref{InducedDkGResult} will be applied implicitly throughout this section.
\begin{remark}
Note that if one applies (i) and (ii) of Proposition \ref{DirectSum1} to the left $D(k[\C{G}])$-module given by Proposition \ref{InducedDkGResult}, the resulting direct sum decomposition is exactly the one from the original $\C{G}$-FA.  Hence, if $(\rho,A)$ is the left $D(k[\C{G}])$-module of Proposition \ref{InducedDkGResult} and $A=\bigoplus_{\mathrm{x}\in \C{G}_0}A^\mathrm{x}$ is the direct sum decomposition of the original $\C{G}$-FA, then the monoidal product $A\widehat{\otimes} A$ of $(\rho,A)$ with itself is $\bigoplus_{\mathrm{x}\in\C{G}_0}A^\mathrm{x}\otimes_k A^\mathrm{x}$ by Lemma \ref{tensor_product2a}.
\end{remark}
\begin{proposition}
\label{InducedAlgebraObject}
If $<\C{G},(A,\m{}{},\textbf{1}_A),\eta,\varphi>$ is a $\C{G}$-FA, then $((\rho,A),m,\mu)$ is a commutative algebra object in $\mbox{Rep}(D(k[\C{G}]))$ where 
\begin{itemize}
\item[(i)] $m:A\widehat{\otimes}A\rightarrow A$ is given by $m(a^\mathrm{x}\otimes b^\mathrm{x}):=\m{a^\mathrm{x}}{b^\mathrm{x}}\in A^\mathrm{x}$, and
\item[(ii)] $\mu: D(k[\C{G}])_t\rightarrow A$ is given by $1^\mathrm{x}\mapsto \textbf{1}_A^\mathrm{x}:=\rho(1^\mathrm{x})\textbf{1}_A\in A^\mathrm{x}_{e_\mathrm{x}}$.
\end{itemize}
\end{proposition}
\begin{proof}
Its clear from the associativity of the $\C{G}$-FA product and the fact that $\textbf{1}_A$ is the multiplicative unit that $m$ and $\mu$ satisfy 1 and 2 of Definition \ref{AlgebraObjectDef}.  

Next, we verify that $m\circ c_{A,A}=m$.  Without loss of generality, take $a=a^\mathrm{x}_g\in A^\mathrm{x}_g$ and $b=b^\mathrm{x}_h\in A^\mathrm{x}_h$.  Then  
\begin{align}
m(a^\mathrm{x}_g\otimes b^\mathrm{x}_h)&=\m{a^\mathrm{x}_g}{b^\mathrm{x}_h}\\
&=\m{(\varphi(g)b^\mathrm{x}_h)}{a^\mathrm{x}_g}\\
&=m(\varphi(g)b^\mathrm{x}_h\otimes a^\mathrm{x}_g)\\
&=m(\rho(\p{ghg^{-1}}{g})b^\mathrm{x}_h\otimes a^\mathrm{x}_g)\\
&=m(\rho(\p{ghg^{-1}}{g})b^\mathrm{x}_h\otimes \rho(\p{g}{e_\mathrm{x}})a^\mathrm{x}_g)\\
&=m\left(\sum_{\mathrm{y}\in \C{G}_0} \sum_{l,m\in \Gamma^\mathrm{y}} \rho(\p{m}{l}) b^\mathrm{x}_h\otimes \rho(\p{l}{e_\mathrm{y}}) a^\mathrm{x}_g\right)\\
&=m\circ c_{A,A}(a^\mathrm{x}_g\otimes b^\mathrm{x}_h)
\end{align}
where the second equality follows from axiom (vii) of Definition \ref{DefGrpdFA}. 

The only thing that remains to be done is to show that $m$ and $\mu$ are $D(k[\C{G}])$-linear.   In the case of $m$, for $x\in \C{G}_1$ with $s(x)=\mathrm{x}$, we have
\begin{align}
\rho(\p{xghx^{-1}}{x}) m(a^\mathrm{x}_g\otimes b^\mathrm{x}_h)&= \varphi(x)(\m{a^\mathrm{x}_g}{b^\mathrm{x}_h})\\
&=\m{(\varphi(x)a^\mathrm{x}_g)}{(\varphi(x)b^\mathrm{x}_h)}\\
&=m\left((\rho(\p{xgx^{-1}}{x})a^\mathrm{x}_g)\otimes (\rho(\p{xhx^{-1}}{x})b^\mathrm{x}_h)\right)\\
&=\sum_{g_1h_1=gh} m\left((\rho(\p{xg_1x^{-1}}{x})a^\mathrm{x}_g)\otimes (\rho(\p{xh_1x^{-1}}{x})b^\mathrm{x}_h)\right)
\end{align}
(where the sum in the last equality is over all $g_1,h_1\in \Gamma^\mathrm{x}$ satisfying $g_1h_1=gh$).  Since the $D(k[\C{G}])$-action on $A\widehat{\otimes} A$ is induced by the coproduct of $D(k[\C{G}])$, the above calculation shows that $m$ is $D(k[\C{G}])$-linear.

In the case of $\mu$, it suffices to show that
\begin{equation}
\label{LeftRightDkG4mu}
\mu(\p{h}{y}\rhd1^\mathrm{x})=\rho(\p{h}{y})\mu(1^\mathrm{x}).
\end{equation}
By (i-b) of Lemma \ref{UnitObject}, the left side is
\begin{equation}
\label{LeftSide1}
\mu(\p{h}{y}\rhd 1^\mathrm{x})=\delta_{s(y),\mathrm{x}}~\delta_{h,yy^{-1}}\mu(1^{s(h)})=\delta_{s(y),\mathrm{x}}~\delta_{h,yy^{-1}}\textbf{1}^{s(h)}_A,
\end{equation}
and the right side is 
\begin{equation}
\label{RightSide1}
\rho(\p{h}{y})\mu(1^\mathrm{x})=\rho(\p{h}{y})\rho(1^\mathrm{x})\textbf{1}_A=\delta_{s(y),\mathrm{x}}\rho(\p{h}{y})\textbf{1}_A.
\end{equation}
Since
\begin{equation}
\nonumber
\textbf{1}_A=\sum_{\mathrm{z}\in\C{G}_0}\textbf{1}^\mathrm{z}_A,
\end{equation}
it follows easily from axioms (i) and (ii) of Definition \ref{DefGrpdFA} and the definition of $\rho$ that $\textbf{1}^\mathrm{z}_A$ is the unit element of $A^\mathrm{z}$ and $\textbf{1}^\mathrm{z}_A\in A^\mathrm{z}_{e_\mathrm{z}}$.  Hence,
\begin{align}
\nonumber
\rho(\p{h}{y})\textbf{1}_A&=\sum_{\mathrm{z}\in \C{G}_0}\rho(\p{h}{y})\textbf{1}^\mathrm{z}_A\\
\nonumber
&=\sum_{\mathrm{z}\in \C{G}_0} \delta_{e_\mathrm{z},y^{-1}hy}\varphi(y)1^\mathrm{z}_A\\
\nonumber
&=\delta_{h,yy^{-1}}\varphi(y)\textbf{1}^{s(y)}_A\\
\label{RightSide2}
&=\delta_{h,yy^{-1}}\textbf{1}^{t(y)}_A
\end{align}
where the last equality follows from the fact that $\varphi(y): A^{s(y)}\rightarrow A^{t(y)}$ is an isomorphism of algebras and must therefore map the unit of $A^{s(y)}$ to that of $A^{t(y)}$.  By substituting (\ref{RightSide2}) into (\ref{RightSide1}) and using the fact that $s(h)=t(y)$, we see that the right side and left side of (\ref{LeftRightDkG4mu}) are indeed equal and this completes the proof.\\
\end{proof}
\begin{notation}
As in the proof of Proposition \ref{InducedAlgebraObject}, we will use $\textbf{1}^\mathrm{x}_A$ to denote the unit element of $A^\mathrm{x}$.  
\end{notation}
\begin{notation}
For a vector space $V$, let $V^\ast$ denote the dual space of $V$, and for a linear map $f:V\rightarrow U$, let $f^\ast: U^\ast\rightarrow V^\ast$ denote the dual of $f$.
\end{notation}
\noindent The next lemma is a technical result which will be used shortly to induce a coproduct on the $\C{G}$-FA.
\begin{lemma}
\label{PsiMap}
Suppose $<\C{G},(A,\m{}{},\textbf{1}_A),\eta,\varphi>$ is a $\C{G}$-FA and $\psi: A\rightarrow A^\ast$ is the $k$-linear map defined by
\begin{equation}
\nonumber
\psi(a)(b):=\eta(a,b)
\end{equation}
where $a,b\in A$ and $\psi(a)\in A^\ast$.  Then
\begin{itemize}
\item[(i)]  $\psi|_{A^\mathrm{x}_g}$ is a vector space isomorphism from $A^\mathrm{x}_g$ to $(A^\mathrm{x}_{g^{-1}})^\ast$, where an element $f$ in $(A^\mathrm{x}_{g^{-1}})^\ast$ is also regarded as an element in $A^\ast$ via $f(b^\mathrm{y}_h)=\delta_{g^{-1},h}~f(b^\mathrm{y}_h)$ for $b^\mathrm{y}_h\in A^\mathrm{y}_h$. 
\item[(ii)] $\psi: A\rightarrow A^\ast$ is a vector space isomorphism; and
\item[(iii)] $\psi\left(\rho(\p{g}{x})a^\mathrm{x}_{x^{-1}gx}\right)=\rho(\p{x^{-1}g^{-1}x}{x^{-1}})^\ast\psi(a^\mathrm{x}_{x^{-1}gx})$.
\end{itemize}
\end{lemma}
\begin{proof}
For (i), the isomorphism from $A^\mathrm{x}_g$ to $(A^\mathrm{x}_{g^{-1}})^\ast$ follows directly from axiom (vi) of Definition \ref{DefGrpdFA}.  The same axiom also implies that $\psi(a^\mathrm{x}_g)(b^\mathrm{y}_h)=0$ when $\mathrm{y}=\mathrm{x}$ and $h\neq g^{-1}$.  For $\mathrm{y}\neq \mathrm{x}$, we have
\begin{align}
\nonumber
\psi(a^\mathrm{x}_g)(b^\mathrm{y}_h)&=\eta(a^\mathrm{x}_g,b^\mathrm{y}_h)\\
\nonumber
&=\eta(\m{a^\mathrm{x}_g}{b^\mathrm{y}_h},\textbf{1}_A)\\
\nonumber
&=0
\end{align}
where the second and third equality follow from axioms (iii) and (ii) of Definition \ref{DefGrpdFA} respectively.  In other words, 
\begin{equation}
\label{psiMap20}
\psi(a^\mathrm{x}_g)(b^\mathrm{y}_h)=\delta_{g^{-1},h}\psi(a^\mathrm{x}_g)(b^\mathrm{y}_h).
\end{equation} 

(ii) is a consequence of part (i) of Lemma \ref{PsiMap} and the fact that $A$ decomposes as $A=\bigoplus_{\mathrm{x}\in \C{G}_0}\bigoplus_{g\in \Gamma^\mathrm{x}}A^\mathrm{x}_g$.

For (iii), let $b^\mathrm{y}_h\in A^\mathrm{y}_h$.  Then we need to show that 
\begin{align}
\label{psiMap1}
\psi\left(\rho(\p{g}{x})a^\mathrm{x}_{x^{-1}gx}\right)(b^\mathrm{y}_h)=\psi(a^\mathrm{x}_{x^{-1}gx})\left(\rho(\p{x^{-1}g^{-1}x}{x^{-1}})b^\mathrm{y}_h\right).
\end{align}
Both sides of (\ref{psiMap1})  are zero for the case when $h\neq g^{-1}$ by (\ref{psiMap20}) and the definition of $\rho$.  For the case when $h=g^{-1}$, we have $\mathrm{y}=t(x)$ and
\begin{align}
\nonumber
\psi\left(\rho(\p{g}{x})a^\mathrm{x}_{x^{-1}gx}\right)(b^\mathrm{y}_{g^{-1}})&=\eta(\rho(\p{g}{x})a^\mathrm{x}_{x^{-1}gx},b^\mathrm{y}_{g^{-1}})\\
\nonumber
&=\eta(\varphi(x)a^\mathrm{x}_{x^{-1}gx},b^\mathrm{y}_{g^{-1}})\\
\nonumber
&=\eta(\varphi(x^{-1})\varphi(x)a^\mathrm{x}_{x^{-1}gx},\varphi(x^{-1})b^\mathrm{y}_{g^{-1}})\\
\nonumber
&=\eta(a^\mathrm{x}_{x^{-1}gx},\varphi(x^{-1})b^\mathrm{y}_{g^{-1}})\\
\nonumber
&=\eta(a^\mathrm{x}_{x^{-1}gx},\rho(\p{x^{-1}g^{-1}x}{x^{-1}})b^\mathrm{y}_{g^{-1}})\\
\nonumber
&=\psi(a^\mathrm{x}_{x^{-1}gx})\left(\rho(\p{x^{-1}g^{-1}x}{x^{-1}})b^\mathrm{y}_{g^{-1}}\right)
\end{align}
where the third equality follows from axiom (iv) of Definition \ref{DefGrpdFA}.
\end{proof}
\noindent In the next two lemmas, we construct the counit and coproduct maps which will give every $\C{G}$-FA the structure of a co-commutative coalgebra object in $\mbox{Rep}(D(k[\C{G}]))$.
\begin{lemma}
\label{counitDkGLinearA}
Suppose $<\C{G},(A,\m{}{},\textbf{1}_A),\eta,\varphi>$ is a $\C{G}$-FA and $\varepsilon:A\rightarrow D(k[\C{G}])_t$ is the $k$-linear map defined by
\begin{equation}
 \varepsilon(a):=\sum_{\mathrm{x}\in \C{G}_0}\eta(a^\mathrm{x},\textbf{1}_A)1^\mathrm{x} 
 \end{equation}
 for $a=\sum_{\mathrm{x}\in \C{G}_0}a^\mathrm{x}$.  Then $\varepsilon$ is $D(k[\C{G}])$-linear.
\end{lemma}
\begin{proof}
It suffices to show that 
\begin{equation}
\label{counitDkGLinear}
\varepsilon\left(\rho(\p{g}{x})a^\mathrm{y}_h\right)=\p{g}{x}\rhd \varepsilon(a^\mathrm{y}_h)
\end{equation}
for $a^\mathrm{y}_h\in A^\mathrm{y}_h$.  

From the definition of $\rho$, we see that the left side of (\ref{counitDkGLinear}) is zero when $\mathrm{y}\neq s(x)$.  Likewise, the right side is also zero when $\mathrm{y}\neq s(x)$ since 
\begin{align}
\nonumber
\p{g}{x}\rhd \varepsilon(a^\mathrm{y}_h)&=\eta(a^\mathrm{y}_h,\textbf{1}_A)\p{g}{x}\rhd 1^\mathrm{y}\\
\nonumber
&=\delta_{s(x),\mathrm{y}}\delta_{g,xx^{-1}} ~\eta(a^\mathrm{y}_h,\textbf{1}_A)1^{t(x)}\\
\nonumber
&=\delta_{s(x),\mathrm{y}}\delta_{g,xx^{-1}}~ \eta(\m{a^\mathrm{y}_h}{\textbf{1}^\mathrm{y}_A},\textbf{1}_A)1^{t(x)}\\
\label{counitDkGLinear1}
&=\delta_{s(x),\mathrm{y}}\delta_{g,xx^{-1}}~ \eta(a^\mathrm{y}_h,\textbf{1}^\mathrm{y}_A)1^{t(x)}
\end{align}
where the second equality follows from part (i-b) of Lemma \ref{UnitObject} and the last equality follows from axiom (iii) of Definition \ref{DefGrpdFA}.

For the case when $\mathrm{y}=s(x)$, the left side of (\ref{counitDkGLinear}) can be rewritten as 
\begin{align}
\nonumber
\varepsilon\left(\rho(\p{g}{x})a^{s(x)}_h\right)&=\eta(\rho(\p{g}{x})a^{s(x)}_h,\textbf{1}_A)1^{t(x)}\\
\nonumber
&=\delta_{x^{-1}gx,h}~\eta(\varphi(x)a^{s(x)}_h,\textbf{1}_A)1^{t(x)}\\
\nonumber
&=\delta_{x^{-1}gx,h}~\eta(\m{(\varphi(x)a^{s(x)}_h)}{\textbf{1}^{t(x)}_A},\textbf{1}_A)1^{t(x)}\\
\nonumber
&=\delta_{x^{-1}gx,h}~\eta(\varphi(x)a^{s(x)}_h,\textbf{1}^{t(x)}_A)1^{t(x)}\\
\nonumber
&=\delta_{x^{-1}gx,h}~\eta(\varphi(x)a^{s(x)}_h,\varphi(x)\textbf{1}^{s(x)}_A)1^{t(x)}\\
\nonumber
&=\delta_{x^{-1}gx,h}~\eta(a^{s(x)}_h,\textbf{1}^{s(x)}_A)1^{t(x)}\\
\label{counitDkGLinear2}
&=\delta_{g,xx^{-1}}~\eta(a^{s(x)}_h,\textbf{1}^{s(x)}_A)1^{t(x)}
\end{align}
where the first equality follows from the fact that $\rho(\p{g}{x})a^{s(x)}_h\in A^{t(x)}$; the sixth equality follows from axiom (iv) of Definition \ref{DefGrpdFA}; and the seventh equality follows from axiom (vi) of Definition \ref{DefGrpdFA} and the fact that $\textbf{1}^{s(x)}_A\in A^{s(x)}_{e_{s(x)}}$.  

By comparing (\ref{counitDkGLinear2}) with (\ref{counitDkGLinear1}), we see that (\ref{counitDkGLinear}) is also satisfied when $\mathrm{y}=s(x)$.
\end{proof}
\begin{lemma}
\label{coproductDkGLinear}
Suppose $<\C{G},(A,\m{}{},\textbf{1}_A),\eta,\varphi>$ is a $\C{G}$-FA.  Let $\psi$ be the map given in Lemma \ref{PsiMap}, $m^{op}:A\widehat{\otimes} A\rightarrow A$ be the $k$-linear map given by $m^{op}(a^\mathrm{x}\otimes b^\mathrm{x}):=\m{b^\mathrm{x}}{a^\mathrm{x}}$, and let $\Delta:A\rightarrow A\widehat{\otimes} A$ be the $k$-linear map given by
\begin{equation}
\nonumber
\Delta:=(\psi^{-1}\otimes \psi^{-1})\circ (m^{op})^{\ast}\circ \psi.
\end{equation}
Then
\begin{itemize}
\item[(i)] $\Delta(a^\mathrm{x}_g)\in \bigoplus_{g_1g_2=g}A^\mathrm{x}_{g_1}\otimes_k A^\mathrm{x}_{g_2}$ for all $a^\mathrm{x}_g\in A^\mathrm{x}_g$, and
\item[(ii)] $\Delta$ is $D(k[\C{G}])$-linear
\end{itemize}
(where the direct sum in (i) is over all $g_1,g_2\in \Gamma^\mathrm{x}$ satisfying $g_1g_2=g$).
\end{lemma}
\begin{proof}
For (i), note that by part (i) of Lemma \ref{PsiMap}, $(m^{op})^\ast\circ \psi(a^\mathrm{x}_g)(b^\mathrm{y}_h\otimes c^\mathrm{y}_l)=0$ for all $b^\mathrm{y}_h\in A^\mathrm{y}_h$ and $c^\mathrm{y}_l\in A^\mathrm{y}_l$ satisfying $lh\neq g^{-1}$.  This implies that 
\begin{equation}
\label{coproductDkGLinear2}
(m^{op})^\ast\circ \psi(a^\mathrm{x}_g)\in \bigoplus_{g_1g_2=g}(A^\mathrm{x}_{g_1^{-1}})^\ast\otimes (A^\mathrm{x}_{g_2^{-1}})^\ast.
\end{equation}
Part (i) of Lemma \ref{coproductDkGLinear} then follows from part (i) of Lemma \ref{PsiMap}.

For (ii), it suffices to show that 
\begin{equation}
\label{coproductDkGLinear1}
\Delta(\rho(\p{g}{x})a^\mathrm{y}_h)=\sum_{g_1g_2=g}[\rho(\p{g_1}{x})\otimes \rho(\p{g_2}{x})]\Delta(a^\mathrm{y}_h)
\end{equation}
for $a^\mathrm{y}_h\in A^\mathrm{y}_h$ (where the sum is over all $g_1,g_2\in \Gamma^{t(x)}$ satisfying $g_1g_2=g$).  From the definition of $\rho$, the left side is zero for $h\neq x^{-1}gx$.  By (i) of Lemma \ref{coproductDkGLinear}, the right side is also zero for $h\neq x^{-1}gx$.  

Let $\mathrm{x}=s(x)$.  For the case when $h=x^{-1}gx$,  we have
\begin{align}
\nonumber
(m^{op})^\ast\circ\psi(\rho(\p{g}{x})a^\mathrm{x}_{x^{-1}gx})&=(m^{op})^\ast\circ \rho(\p{x^{-1}gx}{x^{-1}})^\ast \circ\psi(a^\mathrm{x}_{x^{-1}gx})\\
\label{coproductDkGLinear3}
&=\left[\sum_{g_1g_2=g}\rho(\p{x^{-1}g_1^{-1}x}{x^{-1}})^\ast\otimes \rho(\p{x^{-1}g_2^{-1}x}{x^{-1}})^\ast\right]\circ (m^{op})^\ast\circ \psi(a^\mathrm{x}_{x^{-1}gx})
\end{align}
where the first equality follows from part (iii) of Lemma \ref{PsiMap} and the second equality follows the definition of $\rho$ and the fact that $\varphi(x^{-1})$ is an algebra homomorphism. 

Since
\begin{equation}
\nonumber
(m^{op})^\ast\circ \psi(a^\mathrm{x}_{x^{-1}gx})\in \bigoplus_{g_1g_2=g}(A^\mathrm{x}_{x^{-1}g_1^{-1}x})^\ast\otimes_k (A^\mathrm{x}_{x^{-1}g_2^{-1}x})^\ast
 \end{equation}
 by (\ref{coproductDkGLinear2}), it follows from (i) of Lemma \ref{PsiMap} that 
 \begin{equation}
 \label{coproductDkGLinear4}
 (m^{op})^\ast\circ \psi(a^\mathrm{x}_{x^{-1}gx})=\sum_{g_1g_2=g}\sum_{i=1}^{n[g_1]}\psi(u^\mathrm{x}_{x^{-1}g_1x,i})\otimes \psi(v^\mathrm{x}_{x^{-1}g_2x,i})
 \end{equation}
 for some $u^\mathrm{x}_{x^{-1}g_1x,i}\in A^\mathrm{x}_{x^{-1}g_1x}$ and $v^\mathrm{x}_{x^{-1}g_2x,i}\in A^\mathrm{x}_{x^{-1}g_2x}$.   In particular,
 \begin{equation}
 \Delta(a^\mathrm{x}_{x^{-1}gx})=\sum_{g_1g_2=g}\sum_{i=1}^{n[g_1]}u^\mathrm{x}_{x^{-1}g_1x,i}\otimes v^\mathrm{x}_{x^{-1}g_2x,i}.
 \end{equation}
 
Substituting (\ref{coproductDkGLinear4}) into the right side of (\ref{coproductDkGLinear3}) and applying (iii) of Lemma \ref{PsiMap} as well as the fact that $\rho(\p{s}{y})^\ast\psi(c^\mathrm{z}_t)=0$ for $s\neq t^{-1}$ gives
 \begin{equation}
 \label{coproductDkGLinear5}
 (m^{op})^\ast\circ\psi(\rho(\p{g}{x})a^\mathrm{x}_{x^{-1}gx})=\sum_{g_1g_2=g}\sum_{i=1}^{n[g_1]}\psi(\rho(\p{g_1}{x})u^\mathrm{x}_{x^{-1}g_1x,i})\otimes \psi(\rho(\p{g_2}{x})v^\mathrm{x}_{x^{-1}g_2x,i}).
 \end{equation}
Applying $\psi^{-1}\otimes \psi^{-1}$ to both sides of (\ref{coproductDkGLinear5}) (and using the definition of $\rho$) yields
\begin{align}
\nonumber
\Delta(\rho(\p{g}{x})a^\mathrm{x}_{x^{-1}gx})&=\sum_{g_1g_2=g}[\rho(\p{g_1}{x})\otimes \rho(\p{g_2}{x})]\Delta(a^\mathrm{x}_{x^{-1}gx})
\end{align}
which completes the proof.
 \end{proof}
 \begin{proposition}
 \label{InducedCoalgebraObject}
Suppose $<\C{G},(A,\m{}{},\textbf{1}_A),\eta,\varphi>$ is a $\C{G}$-FA and $\varepsilon$ and $\Delta$ are the maps given in Lemmas \ref{counitDkGLinearA} and \ref{coproductDkGLinear} respectively.  Then $((\rho,A),\Delta,\varepsilon)$ is a co-commutative coalgebra object in $\mbox{Rep}(D(k[\C{G}]))$.
 \end{proposition}
 \begin{proof}
 By Lemmas \ref{counitDkGLinearA} and \ref{coproductDkGLinear}, $\varepsilon$ and $\Delta$ are $D(k[\C{G}])$-linear.  We now verify that $\Delta$ and $\varepsilon$ satisfy the axioms of a co-commutative coalgebra.
 
 For the coassociativity of $\Delta$, we have
 \begin{align}
 (\Delta\otimes id_A)\circ \Delta &=\big[\big[(\psi^{-1}\otimes \psi^{-1})\circ (m^{op})^\ast\big]\otimes \psi^{-1}\big]\circ (m^{op})^\ast\circ \psi\\
 &=\left[\psi^{-1}\otimes \psi^{-1}\otimes \psi^{-1}\right]\circ\big[\big((m^{op})^\ast\otimes id_{A^\ast}\big)\circ (m^{op})^\ast\big]\circ \psi\\
 &=\left[\psi^{-1}\otimes \psi^{-1}\otimes \psi^{-1}\right]\circ\big[\big(id_{A^\ast}\otimes(m^{op})^\ast \big)\circ (m^{op})^\ast\big]\circ \psi\\
 &=\big[\psi^{-1}\otimes \big[(\psi^{-1}\otimes \psi^{-1})\circ (m^{op})^\ast\big]\big]\circ(m^{op})^\ast\circ\psi\\
 &=(id_A\otimes \Delta)\circ \Delta
 \end{align}
 where the third equality is a consequence of the fact that the opposite multiplication map $m^{op}$ of Lemma \ref{coproductDkGLinear} is associative.  
 
 For the counit property, we need to show  that
 \begin{equation}
\label{InducedCoalgebraObject1}
 l_A\circ(\varepsilon\otimes id_A)\circ \Delta(a)=a=r_A\circ(id_A\otimes \varepsilon)\circ \Delta(a)
 \end{equation}
 for all $a\in A$.  By linearity, it suffices to prove (\ref{InducedCoalgebraObject1}) for the case when $a=a^\mathrm{x}_g\in A^\mathrm{x}_g$.  If $a^\mathrm{x}_g$ is zero, there is nothing to prove.  So assume then that $a^\mathrm{x}_g\neq 0$ and let $\{u_j\}_{j=1}^n$ be a basis for $A^\mathrm{x}_{g^{-1}}$ and let $\{v_{i}\}_{i=1}^m$ be a basis for $A^\mathrm{x}_{e_\mathrm{x}}$ where $v_{1}$ is taken to be the projection of $\textbf{1}_A$ onto $A^\mathrm{x}$.  (As was shown in proof of Proposition \ref{InducedAlgebraObject}, $v_1$ is indeed an element of $A^\mathrm{x}_{e_\mathrm{x}}$ and is also the unit element of $A^\mathrm{x}$.)   Furthermore, let $\{u_j^\ast\}_{j=1}^n$ and $\{v_{i}^\ast\}_{i=1}^m$ denote the dual basis of $\{u_j\}_{j=1}^n$ and $\{v_{i}\}_{i=1}^m$ respectively (where an element $f$ in  $(A^\mathrm{y}_h)^\ast$ is also regarded as an element of $A^\ast$ by extending the definition of $f$ via $f(a^\mathrm{z}_l)=\delta_{h,l}f(a^\mathrm{z}_l)$).
 
 By part (i) of Lemma \ref{PsiMap}, we have
\begin{equation}
\label{InducedCoalgebraObject2}
 \psi(a^\mathrm{x}_g)=\sum_{j=1}^n\alpha_ju_j^\ast
 \end{equation}
 where $\alpha_j=\psi(a^\mathrm{x}_g)(u_j)$.  In addition, by part (i) of Lemmas \ref{PsiMap} and \ref{coproductDkGLinear} we can also express $(m^{op})^\ast\circ \psi(a^\mathrm{x}_g)$ as
\begin{equation}
\label{InducedCoalgebraObject3}
(m^{op})^\ast\circ \psi(a^\mathrm{x}_g)=\sum_{i,j}\alpha_{ij}v_i^\ast\otimes u_j^\ast+\omega\in \bigoplus_{g_1g_2=g} (A^\mathrm{x}_{g_1^{-1}})^\ast\otimes_k (A^\mathrm{x}_{g_2^{-1}})^\ast
 \end{equation} 
where $\alpha_{ij}=\psi(a^\mathrm{x}_g)(\m{u_j}{v_i})$ and 
\begin{equation}
\nonumber
\omega\in \bigoplus_{g_1g_2=g, ~g_1\neq e_\mathrm{x}} (A^\mathrm{x}_{g_1^{-1}})^\ast\otimes_k (A^\mathrm{x}_{g_2^{-1}})^\ast.
\end{equation}
In particular, note that $\alpha_j=\alpha_{1j}$.

Next,  note that for $f\in (A^\mathrm{x}_h)^\ast$, we have
\begin{align}
\nonumber
\varepsilon\circ \psi^{-1}(f)&=\eta(\psi^{-1}(f),\textbf{1}_A)~1^\mathrm{x}\\
\nonumber
&=\psi(\psi^{-1}(f))(\textbf{1}_A)~1^\mathrm{x}\\
\nonumber
&=f(\textbf{1}_A)~1^\mathrm{x}\\
\label{InducedCoalgebraObject4}
&=\delta_{h,e_\mathrm{x}} f(\textbf{1}_A)~1^\mathrm{x}.
\end{align}
Applying (\ref{InducedCoalgebraObject3}) and (\ref{InducedCoalgebraObject4}) to the first half of (\ref{InducedCoalgebraObject1}) gives
\begin{align}
 l_A\circ(\varepsilon\otimes id_A)\circ \Delta(a^\mathrm{x}_g)&=l_A\circ(\varepsilon\circ\psi^{-1}\otimes \psi^{-1})\circ (m^{op})^\ast\circ \psi(a^\mathrm{x}_g)\\
 &=\sum_{i,j}\alpha_{ij}~v_i^\ast(\textbf{1}_A)\psi^{-1}(u_j^\ast)\\
 &=\sum_{i,j}\alpha_{ij}~v_i^\ast(v_1)\psi^{-1}(u_j^\ast)\\
 &=\sum_{i,j}\alpha_{1j}\psi^{-1}(u_j^\ast)\\
 &=\sum_{i,j}\alpha_{j}\psi^{-1}(u_j^\ast)\\
 &=a^\mathrm{x}_g.   
\end{align}
 The proof of the other half of (\ref{InducedCoalgebraObject1}) is entirely similar.
 
 Lastly, for co-commutativity, we need to show that
 \begin{equation}
 \label{InducedCoalgebraObject5}
 c_{A,A}\circ \Delta(a)=\Delta(a)\hspace*{0.2in}\forall~a\in A.
 \end{equation}
 Again, by linearity, it suffices to prove (\ref{InducedCoalgebraObject5}) for the case when $a=a^\mathrm{x}_g\in A^\mathrm{x}_g$.  To start, note that by applying $\psi^{-1}$ to both sides of part (iii) of Lemma \ref{PsiMap}  (and using the definition of $\rho$), it follows that 
  \begin{equation}
 \label{InducedCoalgebraObject6}
 \rho(\p{g}{x})\circ\psi^{-1}=\psi^{-1}\circ\rho(\p{x^{-1}g^{-1}x}{x^{-1}})^\ast.
 \end{equation}
 Next, note that if $b^\mathrm{x}_{g_1^{-1}}\in A^\mathrm{x}_{g_1^{-1}}$ and $c^\mathrm{x}_{g_2^{-1}}\in A^\mathrm{x}_{g_2^{-1}}$ with $g_1g_2=g$, then
 \begin{align}
 \nonumber
 (m^{op})^\ast\circ \psi(a^\mathrm{x}_g)(b^\mathrm{x}_{g_1^{-1}}\otimes c^\mathrm{x}_{g_2^{-1}})&=\psi(a^\mathrm{x}_g)(\m{c^\mathrm{x}_{g_2^{-1}}}{b^\mathrm{x}_{g_1^{-1}}})\\
 \nonumber
 &=\psi(a^\mathrm{x}_g)(\m{(\varphi(g_2^{-1})b^\mathrm{x}_{g_1^{-1}})}{c^\mathrm{x}_{g_2^{-1}}})\\
 \label{InducedCoalgebraObject7}
 &=\psi(a^\mathrm{x}_g)(\m{(\rho(\p{g_2^{-1}g_1^{-1}g_2}{g_2^{-1}})b^\mathrm{x}_{g_1^{-1}})}{(\rho(\p{g_2^{-1}}{e_\mathrm{x}})c^\mathrm{x}_{g_2^{-1}}}))
 \end{align} 
 where the second equality follows from axiom (vii) of Definition \ref{DefGrpdFA} and the third equality follows directly from the definition of $\rho$.  (\ref{InducedCoalgebraObject7}) then implies that
 \begin{equation}
 \label{InducedCoalgebraObject8}
 (m^{op})^\ast\circ \psi(a^\mathrm{x}_g)=\left[\sum_{g_1g_2=g}\rho(\p{g_2^{-1}}{g_1^{-1}})^\ast\otimes \rho(\p{g_1^{-1}}{e_\mathrm{x}})^\ast \right]\circ m^\ast\circ\psi(a^\mathrm{x}_g).
 \end{equation}
 
 Now let $\tau:A\widehat{\otimes}A\rightarrow A\widehat{\otimes}A$ be the $k$-linear map defined by $\tau(a^\mathrm{y}\otimes b^\mathrm{y}):=b^\mathrm{y}\otimes a^\mathrm{y}$.  The proof of (\ref{InducedCoalgebraObject5}) then follows from (\ref{InducedCoalgebraObject6}) and (\ref{InducedCoalgebraObject8}):
 \begin{align}
\nonumber
c_{A,A}\circ \Delta(a^\mathrm{x}_g)&=\tau\circ\left[\left(\sum_{\mathrm{y}\in\C{G}_0}\sum_{h,l\in\Gamma^\mathrm{y}}\rho(\p{h}{e_\mathrm{y}})\circ \psi^{-1}\otimes\rho(\p{l}{h})\circ\psi^{-1}\right)\circ (m^{op})^\ast\circ\psi(a^\mathrm{x}_g)\right]\\
\nonumber
&=\tau\circ\left[\left(\sum_{g_1g_2=g}\rho(\p{g_1}{e_\mathrm{x}})\circ \psi^{-1}\otimes\rho(\p{g_1g_2g_1^{-1}}{g_1})\circ\psi^{-1}\right)\circ (m^{op})^\ast \circ\psi(a^\mathrm{x}_g)\right]\\
\nonumber
&=\tau\circ\left[\left(\sum_{g_1g_2=g}\psi^{-1}\circ\rho(\p{g_1^{-1}}{e_\mathrm{x}})^\ast\otimes\psi^{-1}\circ\rho(\p{g_2^{-1}}{g_1^{-1}})^\ast\right)\circ (m^{op})^\ast\circ\psi(a^\mathrm{x}_g)\right]\\
\nonumber
&=(\psi^{-1}\otimes\psi^{-1})\circ\left[\left(\sum_{g_1g_2=g}\rho(\p{g_2^{-1}}{g_1^{-1}})^\ast\otimes\rho(\p{g_1^{-1}}{e_\mathrm{x}})^\ast\right)\circ m^\ast\circ\psi(a^\mathrm{x}_g)\right]\\
\nonumber
&=(\psi^{-1}\otimes\psi^{-1})\circ (m^{op})^\ast\circ\psi(a^\mathrm{x}_g)\\
\nonumber
&=\Delta(a^\mathrm{x}_g)
 \end{align}
 where the third equality follows from (\ref{InducedCoalgebraObject6}) and the fifth equality follows from (\ref{InducedCoalgebraObject8}). 
 \end{proof}
 \noindent The next two lemmas will be used to show that the algebra and coalgebra objects given by Propositions \ref{InducedAlgebraObject} and \ref{InducedCoalgebraObject} satisfy the Frobenius relations (equations (\ref{FrobeniusAxiom1}) and (\ref{FrobeniusAxiom2})).
 \begin{lemma}
 \label{Phi_gInv}
 Suppose $<\C{G},(A,\m{}{},\textbf{1}_A),\eta,\varphi>$ is a $\C{G}$-FA .  Then
   \begin{itemize}
 \item[(i)] $\varphi(g^{-1})|_{A^\mathrm{x}_g}=id_{A^\mathrm{x}_g}$,
  \item[(ii)] $\m{a^\mathrm{x}_g}{b^\mathrm{x}_{g^{-1}}}=\m{b^\mathrm{x}_{g^{-1}}}{a^\mathrm{x}_g}$ for all $a^\mathrm{x}_g\in A^\mathrm{x}_g$, $b^\mathrm{x}_{g^{-1}}\in A^\mathrm{x}_{g^{-1}}$, and
  \item[(iii)] $\eta$ is symmetric
 \end{itemize}
 for all $\mathrm{x}\in \C{G}_0$, $g\in \Gamma^\mathrm{x}$.  
 \end{lemma}
 \begin{proof}
 Part (i) follows immediately from axiom (viii) of Definition \ref{DefGrpdFA} and the fact that $\varphi(e_\mathrm{x})=id_{A^\mathrm{x}}$.  Part (ii) then follows from part (i) of Lemma \ref{Phi_gInv} and axiom (vii) of Definition \ref{DefGrpdFA}.  For (iii), we have
 \begin{align}
 \nonumber
 \eta(a^\mathrm{x}_g,b^\mathrm{x}_h)&=\eta(\m{a^\mathrm{x}_g}{b^\mathrm{x}_h},\textbf{1}_A)\\
 \nonumber
 &=\eta(\m{(\varphi(g)b^\mathrm{x}_h)}{a^\mathrm{x}_g},\textbf{1}_A)\\
 \nonumber
 &=\eta(\varphi(g)b^\mathrm{x}_h,a^\mathrm{x}_g)\\
 \nonumber
 &=\eta(b^\mathrm{x}_h,\varphi(g^{-1})a^\mathrm{x}_g)\\
 \nonumber
 &=\eta(b^\mathrm{x}_h,a^\mathrm{x}_g)
 \end{align}
 where the fourth equality follows from axiom (iv) of Definition \ref{DefGrpdFA} and the last equality follows from part (i) of Lemma \ref{Phi_gInv}.
 \end{proof}
 \begin{lemma}
 \label{uivi}
 Suppose $<\C{G},(A,\m{}{},\textbf{1}_A),\eta,\varphi>$ is a $\C{G}$-FA and $\{u_i\}$ is any basis of $A^\mathrm{x}$ where $u_i\in A^\mathrm{x}_{g_i}$ for some $g_i\in \Gamma^\mathrm{x}$.  Let
 \begin{equation}
\widehat{u}_i:=\psi^{-1}(u_i^\ast)
 \end{equation}
 where $\{u_i^\ast\}$ is the dual basis of $\{u_i\}$.   Then
\begin{itemize}
\item[(i)] $\widehat{u}_i\in A^\mathrm{x}_{g_i^{-1}}$;
\item[(ii)] $\{\widehat{u}_i\}$ is a basis of $A^\mathrm{x}$; and
\item[(iii)] $\psi(u_i)=\widehat{u}_i^\ast$.
\end{itemize} 
 \end{lemma}
 \begin{proof}
 Parts (i) and (ii) are both immediate consequences of Lemma \ref{PsiMap}. By part (iii) of Lemma \ref{Phi_gInv}, we also have 
\begin{align}
\nonumber
\psi(u_i)(\widehat{u}_j)&=\eta(u_i,\widehat{u}_j)\\
\nonumber
&=\eta(\widehat{u}_j,u_i)\\
\nonumber
&=\psi(\widehat{u}_j)(u_i)\\
\nonumber
&=u_j^\ast(u_i)\\
\nonumber
&=\delta_{ij},
\end{align}
which proves (iii).
\end{proof}
\noindent The last two results of this section will be used shortly to establish the first half of Theorem \ref{MainTheorem}.
 \begin{proposition}
 \label{FrobeniusRelation0}
 Suppose $<\C{G},(A,\m{}{},\textbf{1}_A),\eta,\varphi>$ is a $\C{G}$-FA and $((\rho,A),m,\mu)$ and $((\rho,A),\Delta,\varepsilon)$ are the algebra and coalgebra objects given respectively in Propositions \ref{InducedAlgebraObject} and \ref{InducedCoalgebraObject}.  Then $((\rho,A),m,\Delta,\mu,\varepsilon)$ is a Frobenius object in $\mbox{Rep}(D(k[\mathcal{G}]))$.
 \end{proposition}
 \begin{proof}
 The only thing we have left to check are the Frobenius relations:
 \begin{align}
 \label{FrobeniusRelation1}
\Delta\circ m&=(m\otimes id_A)\circ (id_A\otimes \Delta)\\
 \label{FrobeniusRelation2}
 \Delta\circ m&=(id_A\otimes m)\circ (\Delta\otimes id_A)
 \end{align}
 
 To start, let $\{u_i\}$ be any basis of $A^\mathrm{x}$ where $u_i\in A^\mathrm{x}_{g_i}$ for some $g_i\in \Gamma^\mathrm{x}$ and let $\{\widehat{u}_i\}$ be the basis given by Lemma \ref{uivi}.  Then 
 \begin{align}
 \m{u_i}{u_j}&=\sum_{t}C^t_{ij}u_t\\
  \m{\widehat{u}_l}{\widehat{u}_m}&=\sum_{t}\widehat{C}^t_{lm}\widehat{u}_t
 \end{align}
 for some $C^t_{ij},~\widehat{C}^t_{lm}\in k$ where we note that $C^t_{ij}=0$ if $g_t\neq g_ig_j$ and $\widehat{C}^t_{lm}=0$ if $g_t\neq g_mg_l$.  Next, note that
 \begin{equation}
 \label{uHATuProduct}
 \m{\widehat{u}_l}{u_i}=\sum_s C^l_{is}\widehat{u}_s.
 \end{equation}
(\ref{uHATuProduct}) follows from the fact that if $\alpha^s\in k$ is the scalar multiplying $\widehat{u}_s$ then 
\begin{align}
\nonumber
\alpha^s&=\widehat{u}_s^\ast(\m{\widehat{u}_l}{u_i})\\
\nonumber
&=\psi(u_s)(\m{\widehat{u}_l}{u_i})\\
\nonumber
&=\eta(u_s,\m{\widehat{u}_l}{u_i})\\
\nonumber
&=\eta(\m{\widehat{u}_l}{u_i},u_s)\\
\nonumber
&=\eta(\widehat{u}_l,\m{u_i}{u_s})\\
\nonumber
&=\sum_{t} C_{is}^t\eta(\widehat{u}_l,u_t)\\
\nonumber
&=\sum_{t} C_{is}^t \psi(\widehat{u}_l)(u_t)\\
\nonumber
&=\sum_{t} C_{is}^t u_l^\ast(u_t)\\
\nonumber
&=C_{is}^l.
\end{align}
We now prove (\ref{FrobeniusRelation1}).  (The proof of (\ref{FrobeniusRelation2}) is similar.)  To do this, it suffices to show that 
\begin{equation}
\label{FrobeniusRelation3}
 \Delta(\m{u_i}{u_j})=(m\otimes id_A)\circ (id_A\otimes \Delta)(u_i\otimes u_j).
\end{equation}
It follows from the definition of $\Delta$ as well as that of $\{u_i\}$ and $\{\widehat{u}_i\}$ that
 \begin{equation}
\Delta(a^\mathrm{x})=\sum_{l,m}\eta(a^\mathrm{x},\m{\widehat{u}_m}{\widehat{u}_l})~u_l\otimes u_m.
 \end{equation}
 Hence, the right side of (\ref{FrobeniusRelation3}) is 
 \begin{equation}
 \label{FrobeniusRelation4}
 \sum_{l,m}\eta(u_j,\m{\widehat{u}_m}{\widehat{u}_l})~(\m{u_i}{u_l})\otimes u_m.
 \end{equation}
 Computing the left side of (\ref{FrobeniusRelation3}) gives
 \begin{align}
 \nonumber
 \Delta(\m{u_i}{u_j})&=\sum_{l,m}\eta(\m{u_i}{u_j},\m{\widehat{u}_m}{\widehat{u}_l})~u_l\otimes u_m\\
 \nonumber
 &=\sum_{l,m}\eta(\m{\widehat{u}_m}{\widehat{u}_l},\m{u_i}{u_j})~u_l\otimes u_m\\
 \nonumber
 &=\sum_{l,m}\eta(\m{\widehat{u}_m}{(\m{\widehat{u}_l}{u_i})},{u_j})~u_l\otimes u_m\\
 \nonumber
 &=\sum_{l,m}\sum_s C^l_{is}\eta(\m{\widehat{u}_m}{\widehat{u}_s},{u_j})~u_l\otimes u_m\\
 \nonumber
&=\sum_{m,s} \eta(\m{\widehat{u}_m}{\widehat{u}_s},{u_j})~\left(\sum_lC^l_{is}u_l\right)\otimes u_m\\
\nonumber
&=\sum_{m,s} \eta(\m{\widehat{u}_m}{\widehat{u}_s},{u_j})~(\m{u_i}{u_s})\otimes u_m.\\
\nonumber
&=\sum_{m,s} \eta(u_j,\m{\widehat{u}_m}{\widehat{u}_s})~(\m{u_i}{u_s})\otimes u_m.
 \end{align}
 Comparing the last line of the above calculation with (\ref{FrobeniusRelation4}) shows that the left and right sides of (\ref{FrobeniusRelation3}) are indeed equal.
 \end{proof}
 
 \begin{proposition}
 \label{LastResult}
 Suppose $<\C{G},(A,\m{}{},\textbf{1}_A),\eta,\varphi>$ is a $\C{G}$-FA and  $((\rho,A),m,\Delta,\mu,\varepsilon)$ is the Frobenius object of Proposition \ref{FrobeniusRelation0}.  Then $((\rho,A),m,\Delta,\mu,\varepsilon)$ satisfies conditions (1) and (2) of Theorem \ref{MainTheorem}.
 \end{proposition}
 \begin{proof}
 For condition (1), let $a=\sum_{\mathrm{x}\in \C{G}_0}\sum_{g\in \Gamma^\mathrm{x}}a^\mathrm{x}_g$.  Then
\begin{align}
\nonumber
\sum_{\mathrm{x}\in\C{G}_0}\sum_{g\in \Gamma^\mathrm{x}}\rho(\p{g}{g})a&=\sum_{\mathrm{x}\in\C{G}_0}\sum_{g\in \Gamma^\mathrm{x}}\varphi(g)a^\mathrm{x}_g\\
\nonumber
&=\sum_{\mathrm{x}\in\C{G}_0}\sum_{g\in \Gamma^\mathrm{x}}a^\mathrm{x}_g\\
\nonumber
&=a
\end{align}
where the first equality follows from the definition of $\rho$ and the second equality follows from axiom (viii) of Definition \ref{DefGrpdFA}.  

For condition (2), note that 
\begin{equation}
\mbox{Tr}\left(l_c\circ \rho(\p{hgh^{-1}}{h})\right)=\mbox{Tr}\left(l_c\circ \varphi(h)|_{A^\mathrm{x}_g}:A^\mathrm{x}_g\rightarrow A^\mathrm{x}_g\right)
\end{equation}
and 
\begin{equation}
\mbox{Tr}\left(\rho(\p{h}{g^{-1}})\circ l_c\circ \rho(\p{h}{e_\mathrm{x}})\right)=\mbox{Tr}\left(\varphi(g^{-1})\circ l_c|_{A^\mathrm{x}_h}:A^\mathrm{x}_h\rightarrow A^\mathrm{x}_h\right).
\end{equation}
Condition (2) then follows from axiom (ix) of Definition \ref{DefGrpdFA}.
 \end{proof}
 \subsection{Proof of Theorem \ref{MainTheorem}}
 The proof of Theorem \ref{MainTheorem} now follows from Propositions \ref{GFO2GFA} and \ref{FrobeniusRelation0}.  Specifically, Proposition  \ref{GFO2GFA} shows that every Frobenius object in $\mbox{Rep}(D(k[\C{G}]))$ satisfying the two conditions of Theorem  \ref{MainTheorem} induces a $\C{G}$-FA.   This proves the second half of Theorem \ref{MainTheorem}.  In addition, \textbf{every} $\C{G}$-FA is derived from a Frobenius object in $\mbox{Rep}(D(k[\C{G}]))$ which satisfies conditions (1) and (2) of Theorem \ref{MainTheorem}.  To see this, let $\C{A}$ be any $\C{G}$-FA and use Proposition \ref{FrobeniusRelation0} to represent $\C{A}$ as a Frobenius object in $\mbox{Rep}(D(k[\C{G}]))$.  By Proposition \ref{LastResult}, this Frobenius object satisfies conditions (1) and (2) of Theorem \ref{MainTheorem}.  Its easy to check that if Proposition \ref{GFO2GFA} is applied to the aforementioned Frobenius object, the resulting $\C{G}$-FA is exactly $\C{A}$ and this proves the first part of Theorem \ref{MainTheorem}.

 \section{Conclusions $\&$ Directions for Future Work}
 In this paper, we have shown that $\C{G}$-FAs correspond to a certain type of Frobenius object in the representation category of $D(k[\C{G}])$.  This result generalizes an earlier result for group Frobenius algebras \cite{KP}, and, in the process, provides a category-theoretic ``derivation" of the original $\C{G}$-FA axioms introduced in \cite{P}.  Furthermore, when one compares the original $\C{G}$-FA definition (which is quite lengthy) with the category-theoretic statement of Theorem \ref{MainTheorem} (which is quite concise), one can certainly make the case that the natural setting for $\C{G}$-FAs is \textit{categorical} in nature. 
 
 We conclude the paper with the following open questions\footnote{The author wishes to thank the reviewer for his helpful comments and for raising the questions posed here.}:
 \begin{itemize}
\item[1.] Is there a relationship between $\C{G}$-FAs and HQFT (beyond the special case when $\C{G}$ is a finite group)?
\item[2.] Does the notion of a $\C{G}$-FA make sense if $\C{G}$ is replaced by a \textit{category fibered in groupoids} (e.g., Deligne-Mumford stacks)?   
 \end{itemize}
These questions will be explored in a future work.

\end{document}